\pdfoutput=1
\documentclass[english]{ourlematema}
\usepackage{amsmath, amsthm, amssymb, mathtools, hyperref, color}
\usepackage{MnSymbol} 
\usepackage{graphicx}
\usepackage{caption}
\usepackage[all]{xypic}
\usepackage{verbatim}
\usepackage{chemfig, chemnum}
\usepackage{cleveref}
\usepackage{tikz,pgfplots,tikz-3dplot}
\usepackage{tikzscale}
\usetikzlibrary{backgrounds}
\pgfplotsset{compat=1.18}
\usepackage[linesnumbered,lined,commentsnumbered]{algorithm2e}
\usepackage{caption}
\usepackage{subcaption}
\usepackage{float}
\usepackage{makecell}
\usepackage{array}
\usepackage{fancyvrb}

\newtheorem{theorem}{Theorem}
\numberwithin{theorem}{section}
\newtheorem{proposition}[theorem]{Proposition}

\newtheorem{definition}[theorem]{Definition}
\newtheorem{remark}[theorem]{Remark}
\newtheorem{example}[theorem]{Example}
\newtheorem{conjecture}[theorem]{Conjecture}

\theoremstyle{definition}

\newcommand{\RR}{\mathbb{R}}

\newcommand{\CC}{\mathbb{C}}
\newcommand{\ZZ}{\mathbb{Z}}
\newcommand{\NN}{\mathbb{N}}

\newcommand{\var}{\alpha}
\newcommand{\eps}{\varepsilon}
\DeclareMathOperator{\Ann}{Ann}

\renewcommand{\d}{\mathrm{d}}
\newcommand{\F}{\mathcal{F}}

\def\treeConefirst{\tikz[baseline=.1ex]{
\fill (0,0.2) circle (3pt) coordinate (A);
\draw (0,0) node [anchor=north][inner sep=0.75pt]  [font=\scriptsize]  {$X_1$};
\draw (0,0.6) node [anchor=north][inner sep=0.75pt]  [font=\scriptsize]  {$1$};
}
}

\def\treeCone{\tikz[baseline=.1ex]{
\fill (0,0.2) circle (3pt) coordinate (A);
\fill (1,0.2) circle (3pt) coordinate (B);
\draw (0,0) node [anchor=north][inner sep=0.75pt]  [font=\scriptsize]  {$X_1$};
\draw (1,0) node [anchor=north][inner sep=0.75pt]  [font=\scriptsize]  {$X_{2}$};
\draw[thick] (A)--(B);
\draw (0,0.6) node [anchor=north][inner sep=0.75pt]  [font=\scriptsize]  {$1$};
\draw (1,0.6) node [anchor=north][inner sep=0.75pt]  [font=\scriptsize]  {$0$};
}
}

\def\treeCtwofirst{\tikz[baseline=.1ex]{
\fill (0,0.2) circle (3pt) coordinate (A);
\draw (0,0) node [anchor=north][inner sep=0.75pt]  [font=\scriptsize]  {$X_2$};
\draw (0,0.6) node [anchor=north][inner sep=0.75pt]  [font=\scriptsize]  {$2$};
}
}

\def\treeCtwosecond{\tikz[baseline=.1ex]{
\fill (0,0.2) circle (3pt) coordinate (A);
\fill (1,0.2) circle (3pt) coordinate (B);
\draw (0,0) node [anchor=north][inner sep=0.75pt]  [font=\scriptsize]  {$X_1$};
\draw (1,0) node [anchor=north][inner sep=0.75pt]  [font=\scriptsize]  {$X_{2}$};
\draw[thick] (A)--(B);
\draw (0,0.6) node [anchor=north][inner sep=0.75pt]  [font=\scriptsize]  {$-1$};
\draw (1,0.6) node [anchor=north][inner sep=0.75pt]  [font=\scriptsize]  {$2$};
}
}

\def\treeCtwothird{\tikz[baseline=.1ex]{
\fill (0,0.2) circle (3pt) coordinate (A);
\fill (1,0.2) circle (3pt) coordinate (B);
\draw (0,0) node [anchor=north][inner sep=0.75pt]  [font=\scriptsize]  {$X_2$};
\draw (1,0) node [anchor=north][inner sep=0.75pt]  [font=\scriptsize]  {$X_{3}$};
\draw[thick] (A)--(B);
\draw (0,0.6) node [anchor=north][inner sep=0.75pt]  [font=\scriptsize]  {$2$};
\draw (1,0.6) node [anchor=north][inner sep=0.75pt]  [font=\scriptsize]  {$-1$};
}
}

\def\treeCthree{\tikz[baseline=.1ex]{
\fill (0,0.2) circle (3pt) coordinate (A);
\draw (0,0) node [anchor=north][inner sep=0.75pt]  [font=\scriptsize]  {$X_2$};
\draw (0,0.6) node [anchor=north][inner sep=0.75pt]  [font=\scriptsize]  {$1$};
}
}

\title{Algebraic Approaches to\\ Cosmological Integrals}
 
\author{Claudia Fevola}
\address{%
{\small  Université Paris-Saclay, Inria, Palaiseau, France\\
\email{claudia.fevola@inria.fr}}
}

\author{Guilherme L.\ Pimentel}
\address{
{\small Scuola Normale Superiore and INFN, Pisa, Italy\\
\email{guilherme.leitepimentel@sns.it}}
}
\authorline

\author{Anna-Laura Sattelberger}
\address{{\small Max Planck Institute for Mathematics in the Sciences, Leipzig, Germany\\
\email{anna-laura.sattelberger@mis.mpg.de}}
}

\author{Tom Westerdijk}
\address{ {\small Scuola Normale Superiore and INFN, Pisa, Italy\\
\email{tom.westerdijk@sns.it}}}

\authormark{C. Fevola\ - \ G.~L.\ Pimentel \ -\  A.-L. Sattelberger \ -\ T. Westerdijk}

\titlemark{Algebraic Approaches to Cosmological Integrals}

\allowdisplaybreaks

\begin{document}
\maketitle
\begin{abstract}
\noindent Cosmological correlators encode
statistical properties of the initial conditions of our universe. Mathematically, they can often be written as Mellin integrals of  a certain rational function associated to graphs, namely the flat
space wavefunction. The singularities of these cosmological integrals are parameterized by binary hyperplane arrangements. Using different algebraic tools, we shed light on the differential and difference equations satisfied by these integrals. Moreover, we study a multivariate version of partial fractioning of the flat space wavefunction, and propose a graph-based algorithm to compute this decomposition.
\end{abstract}

\section{Introduction}\label{sec:intro}
Cosmological correlation functions are central to the study and characterization of the very early universe. They encode the statistical properties of the initial density inhomogeneities, 
around which structures cluster, giving rise to 
galaxies, stars, planets, etc.
Computing cosmological correlators requires to integrate their time evolution in the early universe. The resulting cosmological integrals are closely analogous to Feynman integrals from particle physics. In this paper, we investigate how various algebraic techniques can shed light on the structure of cosmological integrals and their differential equations.

The cosmological integral is the multi-dimensional Mellin integral
\begin{equation}\label{eq:cosmoint}
    \psi_{\varepsilon}(X,Y)\,=\,\int_{\mathbb{R}^n_{>0}} \psi_{\rm flat}(X+\var,Y) \var^\varepsilon  \, \d  \var  \, ,
\end{equation}
where the integration variable $\var$ parameterizes the time in which a physical process happens.  
Here, $\psi_{\rm flat}$ is the wavefunction in flat space where there is no cosmological time evolution. 
It has been studied extensively in the past few years \cite{cosmowave}. Interestingly, it can be extracted from the canonical form of a cosmological polytope \cite{cosmowave,cosmopol,JKSV23}. For our purposes, it is enough to know that it is a rational function of kinematics, thus resembling the product of propagators that is typical for Feynman integrals. The variables $X$ and $Y$ summarize these kinematics and we will sometimes refer to these variables as ``energies.'' 

An example that already illustrates the challenge of computing these integrals, is the single-exchange
process, with $X_1$ and $X_2$ being the external outgoing energies, while $Y$ is the exchanged energy in the process; the integrand is 
\begin{align*}
    \psi_{\rm flat}(X_1,X_2,Y) \,=\,\frac{2Y}{(X_1+Y)(X_2+Y)(X_1+X_2)} \, .
\end{align*}
The numerator $2Y$ normalizes the residues of the poles of the integrand to be~$\pm 1$, doing justice to its origin as a canonical form.
In this case, as we have two external energies, there will be two integration variables.

Cosmological correlators can be represented by diagrams in spacetime or in kinematic space, some examples of which we show in \Cref{fig:23sitestar}.
\vspace{2mm}
\begin{figure}[h]
\begin{center}
\raisebox{.44cm}{\begin{tikzpicture}[scale=.419, transform shape]
				\tikzstyle{grid lines}=[lightgray,line width=0]
				\tikzstyle{every node}=[circle, draw, fill=black, inner sep=.1pt, minimum width=.1pt] node{};
    \draw[->,black,thick] (-4,-3) -- (-4,1.1);
    \node[draw=none,fill=none] at (-5,-.25) {\huge{time}};
				\draw[black,thick] (-1.5,-1) -- (1.5,-1); 
				\draw[gray!80,very thick] (-3,1) -- (3,1); 
				\draw[gray!80,thick] (-2.4,1) -- (-1.5,-1) -- (-.6,1); 			
				\draw[gray!80,thick] (.6,1) -- (1.5,-1) -- (2.4,1); 
				\node at (-1.5,-1) {X1};
				\node at (1.5,-1) {X2};
				\node[draw=none,fill=none] at  (-1.45,-1.75) {\huge $X_1$};
				\node[draw=none,fill=none] at  (1.55,-1.75) {\huge $X_2$};
				\node[draw=none,fill=none] at  (0,-.55) {\huge $Y$};
\end{tikzpicture}} \quad 
\raisebox{.781cm}{\begin{tikzpicture}[scale=.42, transform shape]
				\tikzstyle{grid lines}=[lightgray,line width=0]
				\tikzstyle{every node}=[circle, draw, fill=black, inner sep=.1pt, minimum width=.1pt] node{};
				\draw[black,thick] (-2.5,-1) -- (2.5,-1); 
				\draw[gray!80,very thick] (-4,1) -- (4,1); 
				\draw[gray!80,thick] (-3.2,1) -- (-2.5,-1) -- (-1.8,1); 		
    	      \draw[gray!80,thick] (1.8,1) -- (2.5,-1) -- (3.2,1); 	
				\draw[gray!80,thick] (0,1) -- (0,-1);
				\node at (-2.5,-1) {X1};
                \node at (0,-1) {X2};
				\node at (2.5,-1) {X3};
				\node[draw=none,fill=none] at  (-2.45,-1.75) {\huge $X_1$};
				\node[draw=none,fill=none] at  (2.55,-1.75) {\huge $X_3$};
                \node[draw=none,fill=none] at  (0,-1.75) {\huge $X_2$};
				\node[draw=none,fill=none] at  (-1.2,-.55) {\huge $Y_{12}$};
                \node[draw=none,fill=none] at  (1.25,-.55) {\huge $Y_{23}$};
\end{tikzpicture}} \quad
\raisebox{0.059cm}{\begin{tikzpicture}[scale=.42, transform shape]
				\tikzstyle{grid lines}=[lightgray,line width=0]
                \node[label=west:\huge{$X_1$},circle, draw, fill=black, inner sep=.1pt, minimum width=.1pt] (X1) at (-2,-.3) {X1};
				\node[label=east:\huge{$X_3$},circle, draw, fill=black, inner sep=.1pt, minimum width=.1pt] (X3) at (2,-.3) {X3};
                \node[label=below:\huge{$X_4$},circle, draw, fill=black, inner sep=.1pt, minimum width=.1pt] (X4) at (0,-1.8) {X4};
                \node[label=west:\huge{$X_2$},circle, draw, fill=black, inner sep=.1pt, minimum width=.1pt] (X2) at (0,.7) {X2};
                \node at (-1.5,-1.35) {\huge{$Y_{14}$}};
				\node at (1.7,-1.35) {\huge{$Y_{34}$}};
                \node at (.6,-.55){\huge{$Y_{24}$}};
                
                \draw[gray!80,very thick] (-4,2) -- (4,2);
                \draw[gray!80,thick] (-3.1,2) -- (X1) -- (-1.7,2);
                \draw[gray!80,thick] (-.7,2) -- (X2) -- (.7,2);
                 \draw[gray!80,thick] (1.7,2) -- (X3) -- (3.1,2);
                \draw[thick,black,label=west:\huge{$Y_{14}$}] (X1) -- (X4) -- (X2);
                 \draw[thick,black] (X3) -- (X4) ;
\end{tikzpicture}}
\end{center}
\vspace{-.5cm}
\caption{A single and a double exchange process, and the $4$-site star graph.}
\label{fig:23sitestar}
\end{figure}
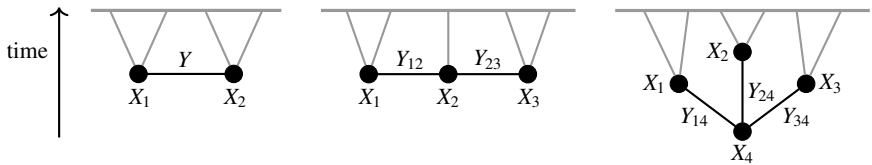

In the figure, time progresses upwards, so the physical picture for the left-most diagram is of a pair of particles being created with energy $Y$, which subsequently decay into two new pairs of particles, with energies $X_1$ and~$X_2$. Their quantum mechanical wavefunction is described by~\eqref{eq:cosmoint}, where we associate a different integrand to each diagram. For practical applications, the $(X,Y)$-variables are real and positive, though it is useful to think of them in larger domains using analytic continuation.

To integrate time, the energies associated to vertices, $X$, have to be shifted by the $\var$-variables. This way, the flat space wavefunction carves up the integration space along hyperplanes. The hyperplane arrangement is
crucial for determining the analytic structure of the integrals. The exponents of $\var$ appearing in the integrand are considered to be the same, single variable $\varepsilon$.
It sets the cosmology in which the process takes place. For example, $\varepsilon=0$ corresponds to an accelerated universe (called ``de Sitter space'' in the physics literature, often a good approximation to cosmic inflation), while $\varepsilon=-1$ recovers the original integrand, as the cosmology will be that of flat space. Other interesting examples are $\varepsilon=-2$ and $\varepsilon=-3$, that parameterize a universe filled with matter and radiation, respectively. Non-integer values of $\varepsilon$ are also interesting for physics, so we consider $\varepsilon$ to be a generic real variable. 

To summarize, cosmological integrals have a structure that is similar to Mellin integrals, with all Mellin variables being equal to $\varepsilon$. The integrand $\psi_{\rm flat}$ is singular only along a hyperplane arrangement that depends on the kinematic variables. This hyperplane arrangement is non-generic, as the coefficients of the integration variables are all zero or one.

Although challenging, the integrals are amenable to many techniques in  
nonlinear algebra, some of which we will establish and apply in this article. We will focus on differential and difference equations for the cosmological correlators. Our goal here is exploratory: we will present a few connections, and open up directions of investigation, posing new questions and conjectures. These connections are promising, and we hope that with further investigation, they will trigger deeper mathematical insights, importing techniques developed in physics, and also enhance the physicist's mathematical toolkit to tackle these fundamental questions.

\medskip

\noindent{\bf Outline.}
\Cref{sec:tools} presents methods from the algebraic theory of linear differential equations as well as their discrete counterpart: shift operators. In \Cref{sec:cosmology}, we apply our methods to cosmological integrals. For the single exchange diagram, we explain the differential equations constructed in \cite{DEcosmological}, in terms of restricted GKZ systems, and moreover systematically construct all recurrence relations for the correlator. Furthermore, we conjecture a multivariate partial fraction decomposition of the flat space wavefunction for arbitrary graphs, which at the same time recovers the physically relevant singularities. \Cref{sec:outlook} summarizes our findings and outlines open questions.

\section{Mathematical tools}\label{sec:tools}
In this section, we introduce the mathematical toolkit that we will use to analyze cosmological integrals.
\subsection{The Weyl algebra}
Let $D_n$ (or just $D$) denote the $n$-th Weyl algebra $D_n=\CC[x_1,\ldots,x_n]\langle \partial_{x_1},\ldots,\partial_{x_n}\rangle$ in the variables $x=(x_1,\ldots,x_n)$. All generators are assumed to commute, except $x_i$ and $\partial_{x_i}$; they obey Leibniz's rule, i.e., $\partial_{x_i}x_i-x_i\partial_{x_i}=1$, $i=1,\ldots,n$. Speaking about $D$-ideals, we will always mean \underline{left} $D$-ideals. We will denote by $R_n=\CC(x_1,\ldots,x_n)\langle \partial_{x_1},\ldots,\partial_{x_n}\rangle$ the rational Weyl algebra. $D_n$ gathers linear differential operators with polynomial coefficients, and $R_n$ allows for coefficients in the field of rational functions. We denote the action of a differential operator $P\in D_n$ on a function $f(x_1,\ldots,x_n)$ as $P\bullet f$. For instance, $\partial_{x_i}\bullet f=\frac{\partial f}{\partial x_i}$, while ${\cdot }$ denotes the product of differential operators, i.e., $\partial_{x_i}\cdot x_i=x_i\partial_{x_i}+1$. 

The {\em holonomic rank} of a $D$-ideal $I$ is the dimension of the quotient $R_n/R_nI$ as a $\CC(x)$-vector space. {\em Holonomic} $D_n$-ideals, i.e., ideals whose characteristic variety is of dimension~$n$, have a finite holonomic rank.  A function $f(x_1,\ldots,x_n)$ is {\em holonomic} if its annihilating $D_n$-ideal $\operatorname{Ann}_{D_n}(f)=\{P\in D_n \,|\, P\bullet f=0\}$ is holonomic as a $D_n$-ideal, see for instance \cite{SatStu19} for a friendly introduction of the~topic.
 
It follows from a theorem of Cauchy, Kovalevskaya, and Kashiwara, that the holonomic rank encodes the dimension of the $\CC$-vector space of holomorphic solutions to the system of PDEs encoded by~$I$---whenever outside $\operatorname{Sing}(I)$, the singular locus of~$I$ (see \cite[Definition 1.12]{SatStu19} for the construction of it).
$D$-ideals of finite holonomic rank do not need to be holonomic, but they can be turned into such by taking the {\em Weyl closure}: The Weyl closure of a $D$-ideal~$I$~is 
\begin{align*}
    W(I) \,=\, R_nI\,\cap\, D_n \,.
\end{align*}
It is again a $D_n$-ideal, it contains~$I$, and it has the same holonomic rank as~$I$.
\begin{remark}
When studying cosmological integrals,
the choice of the most suitable 
tools to use will depend on which parameters of the integrals we vary. The possible choices are explained in detail in~\cite{AFST22}. In \Cref{sec:GKZ_restrictions}, the integrals will be considered as functions of the coefficients of the polynomials defining the integrand, whereas in \Cref{sec:shift_relations}, they will be functions of the exponent~$\eps$. Importantly, the choice of parameters will also determine which Weyl algebra to use. We will specify this choice in each section.
\end{remark}

\subsection{Connection matrices and gauge transformation}
Let $I$ be a $D_n$-ideal of holonomic rank $m$. Let  $(s_1,s_2,\ldots,s_m)$ be a $\CC(x)$-basis of $R_n/R_nI$. The $s_i$'s can be chosen to be monomials in the $\partial_{x_i}$'s, and w.l.o.g., $s_1=1$. By a Gröbner basis reduction of the $\partial_{x_i}s_j$ mod~$I$, one can then read the {\em connection matrices} of $I$. This system is sometimes called a
``Pfaffian system,''\linebreak cf.~\cite[p.~38]{SST00}. For a solution $f$ to~$I$, let $F=(f,s_2\bullet f,\ldots,s_m\bullet f)^\top$. The {connection matrices} of $I$ then are the unique matrices $M_1,\ldots,M_n\in \operatorname{Mat}_{m\times m}(\CC(x))$~s.t.\
\begin{align}\label{eq:connmatrices}
    \partial_{x_i} \bullet F \,=\, M_i\cdot F \,, \qquad i=1,\ldots,n 
\end{align}
for any $f\in \operatorname{Sol}(I)$.
The left-hand side of \eqref{eq:connmatrices} denotes the vector $(\partial_{x_i} \bullet f,(\partial_{x_i}\cdot s_2)\bullet f,\ldots,(\partial_{x_i}\cdot s_m)\bullet f)^\top$. One may equivalently write \eqref{eq:connmatrices} as $\d F=M F$, with $M$ an $m\times m$ matrix of rational differential one-forms. 
Changing basis via $\widetilde{F}=GF$ for some invertible $m\times m$ matrix~$G$ acts via the {\em gauge transformation} on the connection matrices: the transformed system then reads as $\partial_{x_i} \bullet \widetilde{F}=\widetilde{M_i}\cdot\widetilde{F}$ for
\begin{align*}
    \widetilde{M_i} \, =\, GM_iG^{-1}+\frac{\partial G}{\partial x_i}G^{-1} \, ,
\end{align*}
where entry-wise differentiation of the matrix $G$ is meant. We will use gauge transforms in \Cref{sec:twosite} in order to change to a basis consisting of integrals against canonical forms of the bounded regions of the hyperplane arrangement.

\subsection{Restrictions and GKZ systems}\label{sec:GKZ_restrictions}
Later on, we will need restrictions of $D_n$-ideals to subspaces of~$\CC^n$. For now, let $D_n$ denote the Weyl algebra in variables $x_1,\ldots,x_n$. For $m<n$, we will denote by $D_m$ the Weyl algebra $\CC[x_1,\ldots,x_m]\langle \partial_{x_1},\ldots,\partial_{x_m} \rangle$ in the first $m$ variables.
\begin{definition} 
Let $I$ be a $D_n$-ideal. The $D_m$-ideal 
\begin{align}\label{eqrestrictionideal}
	(I+x_{m+1}D_n+\cdots+x_nD_n ) \,\cap \,D_m 
\end{align}
is the {\em restriction ideal} of $I$ to the coordinate subspace $\{x_{m+1}=\cdots=x_n=0\}\subset \CC^n$.
\end{definition}

Let $f(x_1,\ldots,x_n)$ be a holonomic function and $m < n$. Then  the restriction of $f$ to the coordinate subspace $\{x_{m+1}=\cdots=x_n=0\}$, i.e., $f(x_1,\ldots,x_m,0,\ldots,0)$,
is a holonomic function in the variables $x_1,\ldots,x_m$, which is annihilated by the restriction ideal~\eqref{eqrestrictionideal}, cf.\ \cite[Proposition 5.2.4]{SST00}.

In this article, we are interested in restrictions of
GKZ systems, which are prominent examples of $D$-ideals.
They are also called {``$A$-hypergeometric systems,''} and we refer to~\cite{GKZsurvey} for the concise connection between these systems and hypergeometric functions.
They are holonomic systems, implying that their space of holomorphic solutions is finite-dimensional, and their structure is deeply linked to toric geometry and combinatorics. Their singular locus coincides with the zero locus of the principal $A$-determinant, see~\cite[Remark~1.8]{GelfandKapranovZelevinsky}. GKZ systems are encoded by an integer matrix~$A$ and a parameter vector~$\kappa$. Since this is what we need in our study, we here directly address the 
case that $A$ comes from a family of Laurent polynomials. For that, let $\{A_j\}_{j=1,\dots,k}$ be finite subsets of $\ZZ^n$ representing the monomial supports of $k$ Laurent polynomials
\begin{align}\label{eq:Laurentpol}
f_j \, = \, \sum_{u \,\in\, A_j} c_{u,j} \alpha^u \quad \text{with }\quad \alpha^u \,=\, \alpha_1^{u_1}\cdots \alpha_n^{u_n} ,
\end{align}
and parameters $(c_{u,j})_{u\in A_j}$  with values in $\CC^{A_j}= \CC^{|A_j|}$. We build the matrix   
\begin{align*}
A \ = \ \left( \begin{array}{ccc|ccc|c|ccc}
&A_1& & &A_2& &  & &A_k\\
1 & \cdots & 1 & 0 & \cdots & 0 &  & 0 & \cdots & 0 \\
0 & \cdots & 0 & 1 & \cdots & 1 &  &  & \vdots &  \\
& \vdots & & & \vdots & & \cdots& 0 & \cdots & 0\\
0 & \cdots & 0 & 0 & \cdots & 0 & & 1 & \cdots & 1
\end{array} \right) \, ,
\end{align*}
namely the {\em Cayley configuration} of the Laurent polynomials $f_1,\ldots,f_k$,
and consider the Weyl algebra $D_A = \CC[c_u \, |\, u\in A ]\langle\partial_u \, | \, u\in A\rangle$ with variables indexed by the columns of~$A$. The {\em toric ideal} of $A$ is defined to be the binomial ideal 
\begin{align}
I_A \, \coloneqq \, \left\langle \partial^a -\partial^b \,\, | \,\, a-b \in \ker(A), \,\, a,b \in \NN^A \right\rangle  ,
\end{align}
where $a = (a_u)_{u\in A}$ and $\partial^a = \prod_{u\in A}\partial_{u}^{a_u}$, and analogously for $b$. In addition, given a vector of possibly complex parameters $\kappa\in \CC^{n+k}$, we define a second left ideal 
$J_{A,\kappa}$ to be generated by the entries of $A\theta-\kappa$, where $\theta = (\theta_u)_{u\in A}$ and $\theta_u = c_u\partial_u$. Finally, the GKZ system associated to the pair $(A,\kappa)$ is the left \mbox{$D_A$-ideal}
\begin{align*}
H_A(\kappa) \,=\, I_A + J_{A,\kappa} \, .
\end{align*}
A full description of the solutions to such a $D_A$-ideal is given by generalized Euler integrals. We refer to~\cite{GKZ90,AFST22} for a general definition of such integrals and a proof that they provide solutions to GKZ systems, {   and to \cite{BFP14,CGMMMMT22} for convergence criteria for the integral in terms of the Newton polytopes of the Laurent polynomials $f_j$ in~\eqref{eq:Laurentpol}.}

Here, we observe that cosmological integrals as in~\eqref{eq:cosmoint} are examples of generalized Euler integrals, where the coefficients of the polynomials appearing in the integral are restricted to a linear subspace of the full coefficient space $\CC^A = \CC^{A_1}\times \cdots \times \CC^{A_k}$. This fact motivates the need of considering the restriction ideal of $H_A(\kappa)$ when aiming to compute differential equations for~\eqref{eq:cosmoint}. 
Computing GKZ restrictions is a challenging problem in computer algebra. We also point out that they are, in general, not GKZ systems themselves again. We refer to \cite{FFW} for detailed formulae. Recent advances in restriction algorithms for $D$-ideals have been achieved in~\cite{Chestnov}, specifically at the level of their Pfaffian systems~\cite{CGMMMMT22}. In~\cite{FMT24}, based on theoretical and computational considerations, it was conjectured that the singular locus of a restricted GKZ system, arising from a Feynman integral, coincides with the Euler discriminant, cf.~\eqref{eq:Euler_discr}.
In \Cref{sec:cosmology}, we will show that the differential equations for the single-exchange diagram described in \cite{DEcosmological} can be interpreted in terms of a restricted GKZ system. {   The connection between cosmological integrals and restrictions of GKZ systems was recently pioneered in~\cite{GrimmHoefnagels}. Where our works overlap, we find agreement in the results. 
In the study of the two-site chain, we constructed the differential operator $\Delta_3$~\eqref{eq:Delta3} to obtain the correct holonomic rank of our annihilating $D$-ideal; this operator was also found in~\cite[(5.75)]{GrimmHoefnagels}.}

\subsection{Flat space wavefunctions from graphs}
Before delving deeper into the mathematical tools required to study cosmological integrals~\eqref{eq:cosmoint}, we establish the required notation. Let $G = (V,E)$ (or $G = (V(G),E(G))$ to highlight the underlying graph) be an undirected connected graph, where $V$ is a set of $n$ vertices, and $E$ is a collection of edges, namely pairs $ij$ for some $v_i,v_j\in V$. 
We denote by $(X,Y) \coloneqq (X_1,\dots,X_n,\{Y_{ij}\}_{ij\in E})
$ the vector of kinematic parameters associated to each vertex and edge of the graph. This will be the general notation for graphs like the ones drawn in black color within diagrams in spacetime, some of which are presented in \Cref{fig:23sitestar}.

To each graph~$G$, we associate the rational function
\begin{equation}\label{eq:FSwave_explicit}
   \psi_{\text{flat}} \, = \,  \psi_G^{\text{flat}} \, = \, 2^{n-1}\cdot \big(\prod_{ij \in E}Y_{ij}\big) \cdot \frac{P}{\ell_1 \ell_2\cdots \ell_k} \, \in\, \CC(X,Y) \, ,
\end{equation}
where $P,\ell_1,\dots,\ell_k \in \CC[X,Y]$, and  {the normalization factor $2^{n-1}\prod_{ij\in E} Y_{ij}$ arises from the requirement that the residue of a canonical form is $\pm 1$.} More precisely, the polynomial $P$ is conjectured to represent the adjoint hypersurface of the cosmological polytope associated to~$G$~\cite{cosmowave}, and $\ell_i=\ell(H_i)$ for $i=1,\dots,k$ are linear forms that can be computed from the connected subgraphs of $G$. Here, a subgraph is another graph formed from a subset of the vertices and edges of $G$ where all endpoints of the edges of $H$ are in the vertex set of $H$. Given a connected subgraph $H = (V(H),E(H))$ of~$G$, the linear form associated to~$H$ is 
\begin{equation}\label{eq:linear_forms_subgraphs}
\ell\left(H\right) \ =  \sum_{v_i \,\in \, V(H)} \! X_i \ \ \  + \! \sum_{\substack{e \,=\, \{i,j\},\\ v_i\,\in\, V(H), \, v_j \,\nin\, V(H)}} \hspace*{-4mm} \!  Y_{ij} \ \ \ +\! 
    \sum_{\substack{e \,=\, \{i,j\} \,\nin\, E(H),\\ v_i,v_j\,\in\, V(H)}} \hspace*{-3mm} 2Y_{ij} \, ,
\end{equation}
see~\cite{cosmowave,cosmopol} for references. Then, unveiling the use of multi-index notation, the cosmological integral is given as
\begin{align}\label{eq:cosmoint_explicit}
    \psi_{\eps}(X,Y) \, = \, \int_{\RR^n_{>0}} \psi_{\text{flat}}\left(X_1+\var_1,\ldots,X_n+\var_n,\{Y_{ij}\}_{ij\in E}\right)\cdot \var_1^\eps\cdots \var_n^\eps \, \d\var \, ,
\end{align}  
where $\d \var \coloneqq \d \var_1\wedge \dots \wedge \d\var_n.$ In what follows, we denote $L_i \coloneqq \ell_i(X+\alpha,Y)$. 
We will denote the complement of the hyperplane arrangement $(L_1,\ldots,L_k)$ in the algebraic $n$-dimensional torus with coordinates $\alpha=(\alpha_1,\dots,\alpha_n)$ by
\[ \mathcal{H}^G_{(X,Y)}\, = \, \left\{ \alpha\in(\CC^*)^n \,\, : \,\, L_i\left(X,Y;\alpha\right) \neq 0, \ i=1,\dots,k \right\},\]
with $(X,Y)\in \CC^{|V|+|E|}$ generic. 
\Cref{fig:hyperplanestwosite}
shows the hyperplane arrangement arising from the integral associated to the $2$-site chain. 

\subsection{Shift-relations}\label{sec:shift_relations}
The Weyl algebra we will be working with here is
\begin{align*}
    D \,=\, \CC(X,Y)[ \var_1,\ldots, \var_n]\langle \partial_{ \var_1},\ldots,\partial_{ \var_n }\rangle  \, .
\end{align*}
Its elements are linear differential operators, defined globally on affine $n$-space in the $\var$-variables over the field of rational functions in the $X_i$'s and $Y_{ij}$'s.

\begin{lemma}\label{lem:annrat}
    Let $p\in \CC[\var_1,\ldots,\var_n]$ be a polynomial. Let $I$ be the $D_n$-ideal generated by $p\partial_{\alpha_i}+\frac{\partial p}{\partial \var_i}$, $i=1,\ldots,n$. Then the full annihilating $D$-ideal of $1/p$ equals the Weyl closure of~$I$. As a formula,
   $\Ann_{D_n}\left( 1/p\right) \,=\, W(I) . $
\end{lemma}
This statement is an immediate consequence of the following lemma, which we learned from discussions with Uli Walther.
\begin{lemma}\label{lem:Walther}
 Let $p,I$ be as in Lemma~\ref{lem:annrat}. Let $P\in \Ann_D(1/p)$, and denote by $d=\operatorname{ord}(P)$ its order. Then $p^d P\in I$.
\end{lemma}
\begin{proof}
We proceed by an inductive argument on the order of $P$.
If $\operatorname{ord}(P)=0$, then $P\in \CC[\var_1,\ldots,\var_n]$. A polynomial annihilates $1/p$ iff it is zero, and so $p^0P=P\in I$.
Now let $P\in \Ann_D(1/p)$ be of order~$d$. Note that $pP=Pp-[P,p]$. Write $P=q+\sum Q_i \partial_{\var_i}$, where $q$ is a polynomial, and the $Q_i$ are some elements of~$D$. Then $Pp=qp+\sum Q_i \partial_{\var_i} p$.
The sum term is in~$I$. So, modulo~$I$, $qp\equiv Pp$. Thus, $pP=Pp-[P,p]\equiv qp-[P,p]$. 
Since $P$ annihilates $1/p$, so does $Pp-[P,p]=pP$ and hence the same is true for $qp-[P,p]$. This operator is of order at most $d-1$, hence
$p^{d-1}(qp-[P,p])\in I$. So, modulo $I$, $p^dP\in I$.
\end{proof}

To construct shift-relations among cosmological integrals, we will make use of the Mellin transform.
Denote by $\mathcal{S}_n=\CC[\varepsilon]\langle \sigma_\varepsilon,\sigma_\varepsilon^{-1} \rangle $ the shift algebra in the variables $\varepsilon=(\varepsilon_1,\ldots,\varepsilon_n)$. The shifts are w.r.t.\ the variables~$\varepsilon$, i.e., $\sigma_{\varepsilon_i}^{\pm 1}\colon \varepsilon_i \mapsto \varepsilon_i\pm 1$. The generators obey \begin{align*}
\sigma_{\varepsilon_i}^{\pm 1}\varepsilon_i  \,=\,({\varepsilon_i} \pm 1)\sigma_{\varepsilon_i}^{\pm 1} \, .
\end{align*}
This implies $\sigma_{\varepsilon}^a\varepsilon^b=(\varepsilon +a)^b\sigma_{\varepsilon}^a$ for $a\in \ZZ^n$, $b\in \NN^n$. To elements of $D_n$, one associates linear partial differential equations; to elements of $\mathcal{S}_n$, recurrence relations.
We will need to extend the shift algebra~to
\begin{align*}
    \mathcal{S}_n
    \, = \, \CC(X_1,\ldots,X_n,\{ Y_{ij}\}_{ij \in E})[\varepsilon]\langle \sigma_{\varepsilon}^{\pm 1}\rangle  
    \, .
\end{align*}

The algebraic Mellin transform of Loeser--Sabbah \cite{LoeSab} is the isomorphism of the non-commutative $\CC(X_1,\ldots,X_n,\{ Y_{ij}\}_{ij \in E})$-algebras
\begin{align}\begin{split}\label{eq:algmellin}
   \mathfrak{M}\{ \cdot \} \colon \ \CC(X,Y)[\alpha^{\pm 1}]\langle \partial_{ \var_1},\ldots,
   \partial_{ \var_n}\rangle \, \stackrel{\cong}{\longrightarrow} \, \mathcal{S}_n
   \, ,  \  \\
    \var_i^{\pm 1} \mapsto \sigma_{\varepsilon_i}^{\pm 1}, \  \var_i\partial_{ \var_i} \mapsto - \varepsilon_i, \ X_i \mapsto X_i, \ Y_{ij}\mapsto Y_{ij} \, .\, .
   \end{split}
\end{align}
Note that $\mathfrak{M}\{\partial_{ \var_i}\} = -(\varepsilon_i-1)\sigma_{\varepsilon_i}^{-1} $ and $\mathfrak{M}\{ \var_i\theta_{ \var_i}\}=-(\varepsilon_i+1)\sigma_{\varepsilon_i}$, where $\theta_{ \var_i}= \var_i\partial_{ \var_i}$ denotes the $i$-th Euler operator. 
For integration cycles that are suitable for integration by parts with vanishing boundary terms, 
one has $\mathfrak{M}\{P\bullet f\}=\mathfrak{M}\{P\}\bullet \mathfrak{M}\{f\}$. The isomorphism~\eqref{eq:algmellin} implies that
\begin{align*}
\mathfrak{M}\{ \operatorname{Ann}_{D[\var^{\pm 1}]}\left(f\right)\} \,=\, \operatorname{Ann}_{\mathcal{S}_n}\left( \mathfrak{M}\{f\} \right) \, .
\end{align*}
Hence, the Mellin transform of the full annihilator of $\psi_{\text{flat}}$ recovers \underline{all} shift relations in the $\eps$-variables for our cosmological integrals $\mathfrak{M}\{\psi_{\text{flat}} \}$.

\section{Use in cosmology}\label{sec:cosmology}
We here showcase the methods from \Cref{sec:tools} at integrals arising from graphs in cosmology---the path graph on $n$ vertices being called ``$n$-site chain'' there.

\subsection{Two-site chain}\label{sec:twosite}
The linear forms appearing in the denominator of $\psi_2^{\text{flat}}$, after shifting the $X$ variables, are the following elements of $\CC(X_1,X_2,Y)[ \var_1, \var_2]$:
\begin{align*}
        L_1 \,=\,  \var_1 + \var_2 +X_1+X_2 \, , \quad
         L_2 \,=\,  \var_1 +X_1+Y , \quad
         L_3 \,=\,  \var_2 + X_2+ Y  \, .
\end{align*}
The hyperplane arrangement $(L_1,L_2,L_3)$ is depicted in \Cref{fig:hyperplanestwosite}.
\begin{figure}
\begin{center}
\begin{tikzpicture}[scale=.94, transform shape]
\tikzstyle{grid lines}=[lightgray,line width=0]
\tikzstyle{every node}=[circle, draw, fill=black, inner sep=1pt, minimum width=9pt] node{};
\draw[black,->,thick] (-4,0)  -- (1.7,0) ;
\draw[black,->,thick]  (0,-4)  -- (0,1.7) ;		
\draw[black,very thick] (-2,-4)  -- (-2,1.7) ;
\draw[black,very thick] (-4,-2)  -- (1.7,-2) ;
\draw[black,very thick] (-4,1)  -- (1,-4) ;
\node[draw=none,fill=none] at  (.2,.25) {\small $0$};
\node[draw=none,fill=none] at  (2.1,0)  {\small $ \var_1$};
\node[draw=none,fill=none] at  (0,2.0) {\small $ \var_2$};
\node[draw=none,fill=none] at  (-2.13,-4.3) { $\{ L_2=0\}$};
\node[draw=none,fill=none] at  (2.6,-2) {$\{ L_3=0\}$};
\node[draw=none,fill=none] at  (1.8,-4.3) { $\{  L_1=0\}$};
\node[draw=none,fill=none] at  (-1.21,.25) { \small $-(X_1+Y)$};
\node[draw=none,fill=none] at  (-3.7,-.3) {\small $-(X_1+X_2)$};
\node[draw=none,fill=none] at  (.8,-1.75) {\small $-(X_2+Y)$};
\draw[draw=none, fill=blue, fill opacity=0.14] (-2,-1) -- (-2,-2) -- (-1,-2) -- (-2,-1);
\node[draw=none,fill=none] at  (-1.7,-1.7) {\footnotesize $\mathbf{1}$};
\draw[draw=none, fill=gray, fill opacity=0.2] (-3,0) -- (-2,-1) -- (-2,0) -- (-3,0);
\node[draw=none,fill=none] at  (-2.27,-0.32) {\footnotesize $\mathbf{3}$}; 
\draw[draw=none, fill=red, fill opacity=0.16] (-1,-2) -- (0,-3) -- (0,-2) -- (-1,-2);
\node[draw=none,fill=none] at  (-.3,-2.3) {\footnotesize $\mathbf{2}$}; 
\draw[draw=none, fill=teal, fill opacity=0.14] (0,0) -- (-2,0) -- (-2,-1) -- (-1,-2) -- (0,-2) -- (0,0);
\node[draw=none,fill=none] at  (-.92,-.92) {\footnotesize $\mathbf{4}$}; 
\draw[teal!30,very thick,dashed] (-3,0) -- (0,-3) -- (0,0) -- (-3,0);
\end{tikzpicture}
\caption{A real picture of the hyperplane arrangement $(L_1,L_2,L_3)$ for the two-site graph, here depicted for generic $X_1,X_2,Y$ with $X_1,X_2>|Y|$. Together with the coordinate axes, it encloses four bounded regions. To the triangle with label~$i$, we have an associated differential operator $Q_i$ as resulting from~\eqref{eq:gauge}.}
\label{fig:hyperplanestwosite}
\end{center}
\end{figure}
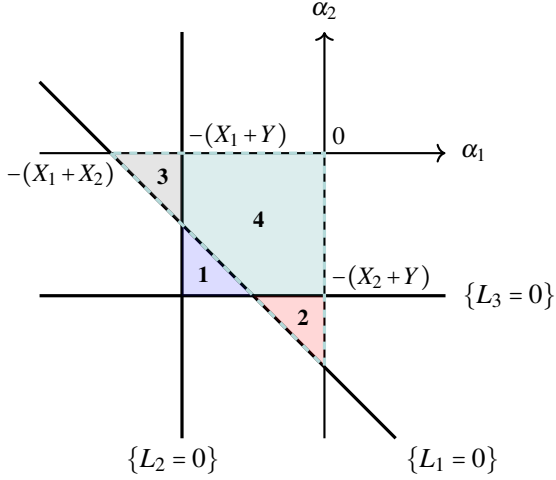

The underlying rational function of interest hence is $\frac{2Y}{L_1L_2L_3}$, and the cosmological integral is the Mellin integral
\begin{align*}
   \int_{\RR_{>0}^2}\frac{2Y }{L_1L_2L_3} \var_1^\varepsilon \var_2^\varepsilon  \, \d  \var_1 \d  \var_2 
   \,=\, 2Y\cdot \mathfrak{M}\left\{ L_1^{-1}L_2^{-1}L_3^{-1}
   \right\} (\varepsilon+1,\varepsilon+1) \, .
\end{align*}

\paragraph{Shift relations}
Denote $L=L_1L_2L_3$. The following four differential operators generate the annihilator of $1/(L_1L_2L_3)$:
\begin{align*}
     P_1  &\,=\,  \var_1 \var_2\partial_{ \var_2}+ \var_2^2\partial_{ \var_2}+(X_2+Y) \var_1\partial_{ \var_2}+(X_1+2X_2+Y) \var_2\partial_{ \var_2} + \var_1+2 \var_2 \\ & \qquad +(X_1+X_2)(X_2+Y)\partial_{ \var_2}+(X_1+2X_2+Y) \, ,\\
 P_2  & \,=\, \var_1 \var_2\partial_{ \var_1}+ \var_2^2\partial_{ \var_2}+(X_2+Y) \var_1\partial_{ \var_1}+(X_1+Y) \var_2\partial_{ \var_1}+2X_2 \var_2\partial_{ \var_2}  +3 \var_2 \\ &\qquad +(X_1+Y)(X_2+Y)\partial_{ \var_1}+(X_2^2-Y^2)\partial_{ \var_2}+(3X_2+Y) \, ,\\
 P_3 &\,=\,  \var_1^2\partial_{ \var_1}- \var_2^2\partial_{ \var_2}+2X_1 \var_1\partial_{ \var_1}-2X_2 \var_2\partial_{ \var_2} +2 \var_1-2 \var_2\\ &\qquad +(X_1^2-Y^2)\partial_{ \var_1}-(X_2^2-Y^2)\partial_{ \var_2}+2(X_1-X_2) \, ,\\
 P_4  &\,=\,  \var_2^2\partial_{ \var_1}\partial_{ \var_2}- \var_2^2\partial_{ \var_2}^2+2X_2 \var_2\partial_{ \var_1}\partial_{ \var_2}-2X_2 \var_2\partial_{ \var_2}^2+2 \var_2\partial_{ \var_1}-4 \var_2\partial_{ \var_2} \\ &\qquad +(X_2^2-Y^2)\partial_{ \var_1}\partial_{ \var_2}-(X_2^2-Y^2)\partial_{ \var_2}^2+2X_2\partial_{ \var_1}-4X_2\partial_{ \var_2}-2 \, .
\end{align*}
This $D$-ideal can be computed with the {\sc Singular:Plural}~\cite{Singular,Plural} library {\tt dmod\_lib}~\cite{Dmodlib}, or, equivalently, by computing the Weyl closure of the \mbox{$D$-ideal} generated by $L\partial_{ \var_i}-\partial_{ \var_i}\bullet L$, where $i=1,2$. 

\begin{remark}
The Bernstein--Sato polynomial of $L=L_1L_2L_3$ is $b_L=(s+1)^2$. Since $-1\notin V(b_L)+\ZZ_{>0}$,
one has the following relation to the $s$-parametric annihilator of $L$: $\Ann_{D_2[s_1,s_2,s_3]}(L_1^{s_1}L_2^{s_2}L_3^{s_3})|_{(s_1,s_2,s_3)=(-1,-1,-1)}=\Ann_{D_2}(1/L)$, cf.\ \cite[Theorem~2]{OakuLog}. A multivariate analogue is given in~\cite[Proposition~3.6]{OT99}. 
If one instead allows individual powers of the hyperplanes, one needs to pass on to Bernstein--Sato ideals. For some well-behaved cases, generators of parametric annihilators can be expressed in terms of logarithmic derivations, cf.~\cite{BathSaito22}.
\end{remark}

The left ideal in the shift algebra $\mathcal{S}_2$ that is generated by $\mathfrak{M}\{ P_1\},\ldots,\mathfrak{M}\{ P_4\}$ encodes all shift relations for $\mathfrak{M}\{1/L\}(\varepsilon_1,\varepsilon_2)$. We here display a selection of the obtained recurrence relations:
\begin{align*}
   \mathfrak{M}\{P_1\} & \,=\, -(\varepsilon_2 - 1)\big[(X_2 + Y)\sigma_1\sigma_2^{-1} + \sigma_1 + \sigma_2 \nonumber \\ 
   & \quad  + (X_1 + X_2)(X_2 + Y)\sigma_2^{-1} + (X_1 + 2X_2 + Y)\big] \, ,\\ \nonumber
    \mathfrak{M}\{P_2\}+\mathfrak{M}\{P_3\} & \,=\, -(\varepsilon_1 - 1)\big[  (X_1 + Y)\sigma_1^{-1} \sigma_2 + \sigma_1+ \sigma_2 \\
    &\quad+ (X_1 + X_2)(X_1 + Y)\sigma_1^{-1} + (2X_1 + X_2 + Y)\big] \, . \nonumber
\end{align*}
Note that the shift operators mirror the symmetry of the integral when swapping the variables $X_1$ and $X_2$. The remaining operators can be computed analogously. 

\begin{remark}
To obtain shift relations among master integrals involving only $\eps$ and $\eps-1$ as in \cite[(3.44)]{DEcosmological} for $\eps=\eps_1=\eps_2$, one will have to combine our shift operators with the  Pfaffian system \cite[(3.30)]{DEcosmological} arising from the canonical forms.
\end{remark}

\paragraph{GKZ system}
We now consider the cosmological integral for the two-site chain as a generalized Euler integral and set up its GKZ system. For that, we leave the coefficients of the linear forms generic, i.e., we consider the integral
\begin{align*}
    2Y \cdot  \int_{\Gamma}  
    \frac{1}{(c_1\alpha_1+c_2\alpha_2+c_3)(c_4 \alpha_1 + c_5)(c_6 \alpha_2 + c_7)}\, \alpha_1^{\epsilon+1} \alpha_2^{\epsilon+1}\,\frac{\d \alpha_1}{\alpha_1} \frac{\d \alpha_2}{\alpha_2} \,  .
\end{align*}
{   Note that, for generalized Euler integrals, $\Gamma$ is taken to be a twisted $n$-cycle, cf.~\cite[Section~4]{AFST22}. However, in physics applications, one instead integrates over the positive orthant.}
The toric ideal corresponding to the integral above is 
\begin{align*}
    I_A \,=\, \left\langle \partial_{c_2} \partial_{c_7} - \partial_{c_3} \partial_{c_6} , \, \partial_{c_1} \partial_{c_5} - \partial_{c_3} \partial_{c_4} \right\rangle ,
\end{align*}
and the parameter vector $\kappa=(-(\eps+1),-(\eps+1),-1,-1,-1)^\top$, cf.~\cite[Section~4]{AFST22}. 
The ideal $ J_{A,\kappa}$ is generated by the operators 
\begin{align*}\begin{split}
 \theta_{c_1} + \theta_{c_4} +  (\eps +1 ), \ 
    \theta_{c_2} + \theta_{c_6} + (\eps + 1), \quad \ \  \\
    \theta_{c_1} + \theta_{c_2} + \theta_{c_3} + 1,  \
    \theta_{c_4} + \theta_{c_5} + 1, \  
    \theta_{c_6} + \theta_{c_7} + 1 \, .
\end{split}\end{align*}
Together, they yield the $D$-ideal $H_A(\kappa)=I_A + J_{A,\kappa}$. We now restrict $H_A(\kappa)$ to the subspace $\{ c\in \CC^7 \, | \, c_1 = c_2 =  c_4 = c_6 = 1 \}$, and moreover change variables from the remaining $c$'s to $(X_1,X_2,Y)$ via $
c_3 = X_1+X_2$, $c_5 = X_1 + Y$, and $c_7 = X_2 + Y$.
The resulting $D$-ideal, $H_A^{\text{res}}(\kappa)$, has holonomic rank~$4$. 
This $D$-ideal was reconstructed purely combinatorially from the line arrangement $\{L_1,L_2,L_3\}$ in \cite{PfiSat}.

We compare the restricted GKZ system with a $D$-ideal which originates from operators in the variables $X_1,X_2,Y$ considered in~\cite{DEcosmological}.  Let 
\begin{align*}
\begin{split}
    \Delta_{1} \,=\, (X_1^2-Y^2)\partial_{X_1}^2 + 2(1-\varepsilon)X_1\partial_{X_1} - \varepsilon(1-\varepsilon) \, , \\ \Delta_{2} \,=\, (X_2^2-Y^2)\partial_{X_2}^2 + 2(1-\varepsilon)X_2\partial_{X_2} - \varepsilon(1-\varepsilon) \, .
\end{split}\end{align*}
These operators are right factors of the operators considered in~\cite[Section~3.3]{DEcosmological}.
The $D$-ideal generated by these, however, has infinite holonomic rank. We extend it by using the operator
\begin{align}\label{eq:Delta3}
    \Delta_3 \,=\,  (X_1+Y)(X_2+Y)\partial_{X_1}\partial_{X_2} - \epsilon(X_1+Y)\partial_{X_1} - \epsilon(X_2+Y)\partial_{X_2} + \epsilon^2 \, ,
\end{align}
and the homogeneity operator 
\begin{align*}
    H &\,=\, X_1\partial_{X_1} + X_2\partial_{X_2} + Y\partial_Y - 2\varepsilon \, .
\end{align*}
Denote by $I$ the $D$-ideal 
\begin{align}\label{eq:Itwosite}
    I \,=\, \left\langle \Delta_1 + \Delta_3,\,  \Delta_2 + \Delta_3, \,  H \right\rangle \, .
\end{align}
This $D$-ideal in \eqref{eq:Itwosite} annihilates the correlation function. Its singular locus is
\begin{align*}
\operatorname{Sing}(I) \,=\, V \left( Y(X_1+X_2)(X_1+Y)(X_1-Y)(X_2+Y)(X_2-Y) \right)  ,
\end{align*}
which corresponds exactly to the singular locus of $H_A^{\text{res}}(\kappa)$.
The holonomic rank of $I$ is~$4$. 

As a $\CC(X_1,X_2,Y)$-basis of $R_3/R_3I$, we choose $( 1,\partial_{X_1},\partial_{X_2},\partial_{X_1}\partial_{X_2} )$. 
We have the following relation, which we proved computationally using {\tt Singular} and the package {\tt HolonomicFunctions}~\cite{HolFun} in Mathematica. We provide our code
via GitLab at \href{https://uva-hva.gitlab.host/universeplus/algebraic-approaches-to-cosmological-integrals}{https://uva-hva.gitlab.host/universeplus/algebraic-approaches-to-cosmological-integrals}.
\begin{proposition}\label{prop:GKZres2site}
Let $I$ be the $D$-ideal in~\eqref{eq:Itwosite}.
The $D$-ideal $DIY=\{ PY \,|\, P\in I\}$ is contained in the restricted GKZ system, $DIY \subset H_A^{\text{res}}(\kappa)$.
Moreover, the two $D$-ideals coincide when considered as ideals in the rational Weyl algebra. 
\end{proposition}
In particular, the proposition implies that the Weyl closures of the two \mbox{$D$-ideals} coincide, i.e., $W(H_A^{\text{res}}(\kappa)) = W(DIY)$.

\paragraph{Canonical forms}
Denote by $\Omega_i$, $i=1,\ldots,4$ the canonical form~\cite{PosGeom} of the triangle labelled by $i$ in \Cref{fig:hyperplanestwosite}. For the vector of integrals
\begin{align}\label{eq:canint}
    \left(\int_{\RR_{>0}^2} x_1^\varepsilon x_2^\varepsilon \, \Omega_1,\, \ldots,\int_{\RR_{>0}^2} x_1^\varepsilon x_2^\varepsilon  \, \Omega_4\right) \, ,
\end{align}
the resulting matrix differential system is in $\eps$-factorized form, cf.~\cite[(3.30)]{DEcosmological}. Denoting the first entry in \eqref{eq:canint} by~$\psi$, the full vector \eqref{eq:canint} can then be written as $(Q_1\bullet \psi,\ldots,Q_4\bullet \psi)^\top$  for $(Q_1,Q_2,Q_3,Q_4)=G\cdot (1,\partial_{X_1},\partial_{X_2},\partial_{X_1}\partial_{X_2})^\top$, where $G$ denotes the matrix 
{\small 
\begin{align}\label{eq:gauge}
\frac{1}{2Y\eps^2}\cdot \begin{pmatrix}
 2 Y \epsilon ^2 & -\epsilon ^2 \left(X_1-Y\right) & -\epsilon ^2 \left(X_2-Y\right) &  -\eps^2 \left(X_1+X_2\right)  \\
 0 & \eps  \left(X_1^2-Y^2\right)  & 0 & \eps \left(X_1+X_2\right)  \left(X_1+Y\right) \\
 0 & 0 & \epsilon  \left(X_2^2-Y^2\right) & \eps \left(X_1+X_2\right)  \left(X_2+Y\right) \\
 0 & 0 & 0 & -\left(X_1+X_2\right) \left(X_1+Y\right) \left(X_2+Y\right)
\end{pmatrix}^\top \! .
\end{align}
}

\noindent This is the gauge matrix for changing from the basis $(1,\partial_{X_1},\partial_{X_2},\partial_{X_1}\partial_{X_2})$ of $R_3/R_3I$ to the basis obtained from canonical forms, which yields the $\eps$-factorized form of the connection matrix.

\subsection{Partial fractioning of the flat space wavefunction}
Partial fraction decomposition is a widely used technique in particle physics and cosmology for simplifying the computation of Feynman integrals, cosmological integrals, and their differential equations. Although standard partial fraction decomposition applies to univariate rational functions, most expressions arising in physics are multivariate. New algorithmic methods for multivariate decompositions were presented in~\cite{HellerManteuffel}. We introduce a formula for the partial fraction decomposition of the wavefunction in flat space associated to a graph from~\eqref{eq:FSwave_explicit}, inspired by the Feynman rules for cosmology as presented in \cite[Section~2]{DEcosmological}. Our approach is entirely combinatorial and can be described in terms of  subgraphs of the original graph. We begin by introducing some definitions. 

Let $H$ be an oriented connected graph with $m>0$  vertices. Given a vertex $v\in V(H)$, we define its {degree} $\deg(v)$ to be the difference between the number of outgoing and incoming edges connected to~$v$. The degree of the graph~$H$ is $\deg(H)=\sum_{v\in V(H)}\deg(v)$. When we consider a subgraph of~$H$, we label the vertices in the subgraph by the degree of the 
vertices in~$H$ and, in slight abuse of notation, refer to these labels as ``degrees'' again. Notice that, in contrast to~$H$, which is always of degree $0$, the subgraphs can have a non-zero degree in general. 
We denote by $(H)^+$ the set of connected
subgraphs of $H$ which have degree strictly positive. We will denote the elements of $(H)^+$ by $H_i^+$. We write $H'\leq H$ if the graph $H'$ is a subgraph of~$H$.
\begin{definition}
    A tuple of $m-1$ indices $(i_1,\dots,i_{m-1})$ is said to be \textit{admissible} if for all $i_k,i_j\in \{i_1,\dots,i_{m-1}\}$, the subgraphs $H^+_{i_j}$ and $H^+_{i_k}$ do not share any vertex
    or one is contained in the other, i.e.,  $H^+_{i_k}\leq H^+_{i_j}$ or $H^+_{i_j}\leq H^+_{i_k}$. 
\end{definition}
\noindent
To a graph $H$ as above, we associate the rational function
\begin{equation}\label{eq:rational_function}
    R_{H}(X,Y) \,\coloneqq\, 
    \begin{cases}
    \frac{1}{\ell(H)} \quad\quad &\text{if } m=1,\\
    \frac{1}{\ell(H)} \cdot  \sum_{\substack{({i_1},\dots,i_{m-1})\\
    \text{admissible}}} \, \frac{1}{\ell(H^+_{i_1})\cdots \ell(H^+_{i_{m-1}})}  \quad\quad &\text{if } m>1,\\
    \end{cases}
\end{equation}
and $R_{H}(X,Y)=0$ if $H$ includes a directed cycle. In other words, $R_H(X,Y)$ is only non-zero if $H$ is a directed acyclic graph. Here, for a subgraph $H$, the linear form $\ell(H)$ is the one defined in~\eqref{eq:linear_forms_subgraphs}. We now have all the necessary tools to describe the decomposition of the flat space wavefunction. 

Let $G=(V,E)$ be a 
graph as in \Cref{sec:tools}, where $n$ denotes the number of vertices. 
We denote by $\mathcal{F}=\mathcal{F}(G) = \{H_1,H_2,\dots,H_r\}$ the set of spanning subgraphs of $G$, where we also take orientation into account. Since each~$H_i$ may be disconnected, we denote by $H_{i}^1,\ldots,H_i^{s_i}$ its connected components, i.e., the maximal connected subgraphs of~$H_i$.

\begin{conjecture}\label{conj1}
Let $G=(V,E)$ be a Feynman graph as above. Then the flat space wavefunction can be decomposed into partial fractions of the form 
\begin{align}\label{eq:PF_conj}
\psi_{\text{flat}}\, = \, \sum_{H_i \,\in\, \mathcal{F}} (-1)^{n-1-|E(H_i)|}\cdot\prod_{j=1}^{s_i}  R_{H_i^j}(X,Y) \, .
\end{align}
\end{conjecture}
We believe that some kind of minimality---and in that sense uniqueness---may be stated for such a decomposition.
Specifically, the number of linear factors in the denominator of each term in the decomposition always coincides with the number of variables $X_i$ associated to vertices, and the degree of the numerator is always zero. On that note, we expect~\eqref{eq:PF_conj} to be the finest decomposition that does not introduce any new hyperplanes in the denominators.

\begin{remark}
    Note that the condition yielding $R_{H_i}=0$ is never satisfied when $G$ is a tree. To obtain $R_{H_i}=0$ for some~$H_i$, one needs a graph of genus at least~$1$, i.e., containing at least one ``loop'' in physics terminology. Some of such graphs are shown in \Cref{fig:loops}. 
\begin{figure}
  \begin{center}
\tikzset{every picture/.style={line width=0.75pt}} 
\begin{tikzpicture}[x=0.75pt,y=0.75pt,yscale=-1,xscale=1,scale=0.8]

\draw  [fill={rgb, 255:red, 0; green, 0; blue, 0 }  ,fill opacity=1 ] (151,225) .. controls (151,222.79) and (152.79,221) .. (155,221) .. controls (157.21,221) and (159,222.79) .. (159,225) .. controls (159,227.21) and (157.21,229) .. (155,229) .. controls (152.79,229) and (151,227.21) .. (151,225) -- cycle ;
\draw  [fill={rgb, 255:red, 0; green, 0; blue, 0 }  ,fill opacity=1 ] (191,225) .. controls (191,222.79) and (192.79,221) .. (195,221) .. controls (197.21,221) and (199,222.79) .. (199,225) .. controls (199,227.21) and (197.21,229) .. (195,229) .. controls (192.79,229) and (191,227.21) .. (191,225) -- cycle ;
\draw   (172,201) -- (178,204.75) -- (172,208.5) ;
\draw   (177.98,248.51) -- (172,244.74) -- (178.02,241.01) ;
\draw   (155,225) .. controls (155,213.95) and (163.95,205) .. (175,205) .. controls (186.05,205) and (195,213.95) .. (195,225) .. controls (195,236.05) and (186.05,245) .. (175,245) .. controls (163.95,245) and (155,236.05) .. (155,225) -- cycle ;
\draw  [fill={rgb, 255:red, 0; green, 0; blue, 0 }  ,fill opacity=1 ] (240,225) .. controls (240,222.79) and (241.79,221) .. (244,221) .. controls (246.21,221) and (248,222.79) .. (248,225) .. controls (248,227.21) and (246.21,229) .. (244,229) .. controls (241.79,229) and (240,227.21) .. (240,225) -- cycle ;
\draw  [fill={rgb, 255:red, 0; green, 0; blue, 0 }  ,fill opacity=1 ] (280,225) .. controls (280,222.79) and (281.79,221) .. (284,221) .. controls (286.21,221) and (288,222.79) .. (288,225) .. controls (288,227.21) and (286.21,229) .. (284,229) .. controls (281.79,229) and (280,227.21) .. (280,225) -- cycle ;
\draw   (267.12,208.4) -- (261,204.85) -- (266.88,200.9) ;
\draw   (260.99,241.01) -- (267,244.74) -- (261.01,248.51) ;
\draw   (244,225) .. controls (244,213.95) and (252.95,205) .. (264,205) .. controls (275.05,205) and (284,213.95) .. (284,225) .. controls (284,236.05) and (275.05,245) .. (264,245) .. controls (252.95,245) and (244,236.05) .. (244,225) -- cycle ;
\draw  [fill={rgb, 255:red, 0; green, 0; blue, 0 }  ,fill opacity=1 ] (429,225) .. controls (429,222.79) and (430.79,221) .. (433,221) .. controls (435.21,221) and (437,222.79) .. (437,225) .. controls (437,227.21) and (435.21,229) .. (433,229) .. controls (430.79,229) and (429,227.21) .. (429,225) -- cycle ;
\draw  [fill={rgb, 255:red, 0; green, 0; blue, 0 }  ,fill opacity=1 ] (469,225) .. controls (469,222.79) and (470.79,221) .. (473,221) .. controls (475.21,221) and (477,222.79) .. (477,225) .. controls (477,227.21) and (475.21,229) .. (473,229) .. controls (470.79,229) and (469,227.21) .. (469,225) -- cycle ;
\draw   (456.12,208.4) -- (450,204.85) -- (455.88,200.9) ;
\draw   (449.99,241.01) -- (456,244.74) -- (450.01,248.51) ;
\draw   (433,225) .. controls (433,213.95) and (441.95,205) .. (453,205) .. controls (464.05,205) and (473,213.95) .. (473,225) .. controls (473,236.05) and (464.05,245) .. (453,245) .. controls (441.95,245) and (433,236.05) .. (433,225) -- cycle ;
\draw    (433,225) .. controls (454,207) and (461,216) .. (473,225) ;
\draw    (433,225) .. controls (456,244) and (466,229) .. (473,225) ;
\draw   (449.99,231.01) -- (456,234.74) -- (450.01,238.51) ;
\draw   (456.12,218.4) -- (450,214.85) -- (455.88,210.9) ;
\draw   (358,209) -- (384,244.5) -- (332,244.5) -- cycle ;
\draw  [fill={rgb, 255:red, 0; green, 0; blue, 0 }  ,fill opacity=1 ] (354,209) .. controls (354,206.79) and (355.79,205) .. (358,205) .. controls (360.21,205) and (362,206.79) .. (362,209) .. controls (362,211.21) and (360.21,213) .. (358,213) .. controls (355.79,213) and (354,211.21) .. (354,209) -- cycle ;
\draw  [fill={rgb, 255:red, 0; green, 0; blue, 0 }  ,fill opacity=1 ] (329,244.5) .. controls (329,242.29) and (330.79,240.5) .. (333,240.5) .. controls (335.21,240.5) and (337,242.29) .. (337,244.5) .. controls (337,246.71) and (335.21,248.5) .. (333,248.5) .. controls (330.79,248.5) and (329,246.71) .. (329,244.5) -- cycle ;
\draw  [fill={rgb, 255:red, 0; green, 0; blue, 0 }  ,fill opacity=1 ] (380,244.5) .. controls (380,242.29) and (381.79,240.5) .. (384,240.5) .. controls (386.21,240.5) and (388,242.29) .. (388,244.5) .. controls (388,246.71) and (386.21,248.5) .. (384,248.5) .. controls (381.79,248.5) and (380,246.71) .. (380,244.5) -- cycle ;
\draw   (342.2,224.91) -- (348.79,222.35) -- (348.21,229.4) ;
\draw   (371.58,220.21) -- (371.6,227.29) -- (365.23,224.21) ;
\draw   (360.99,248.51) -- (355,244.74) -- (361.01,241.01) ;

\draw (145,205.4) node [anchor=north west][inner sep=0.75pt]    {$0$};
\draw (196,206.4) node [anchor=north west][inner sep=0.75pt]    {$0$};
\draw (234,205.4) node [anchor=north west][inner sep=0.75pt]    {$0$};
\draw (285,206.4) node [anchor=north west][inner sep=0.75pt]    {$0$};
\draw (423,205.4) node [anchor=north west][inner sep=0.75pt]    {$0$};
\draw (474,206.4) node [anchor=north west][inner sep=0.75pt]    {$0$};
\draw (325,223.4) node [anchor=north west][inner sep=0.75pt]    {$0$};
\draw (363,197.4) node [anchor=north west][inner sep=0.75pt]    {$0$};
\draw (381,223.4) node [anchor=north west][inner sep=0.75pt]    {$0$};
\end{tikzpicture}
\caption{Some examples of graphs with $R_{H_i}=0$.}
\label{fig:loops}
\end{center}
\end{figure}
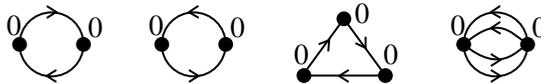
\end{remark}

\begin{example}[$3$-site chain] Here, we demonstrate our procedure at the example of the flat space wavefunction of the $3$-site chain in detail. An illustration of the graph with the labeling of the kinematic parameters can be found in \Cref{fig:23sitestar}. The linear forms associated to its connected subgraphs are
\begin{align}\label{eq:ell_i_3site}
    \ell_1 & \,=\, X_1+Y_{12} \, , & \ell_2 &\,=\, X_2 + Y_{12} + Y_{23}\,, &  \ell_3 &\,=\, X_3  + Y_{23} \, ,\\ 
    \ell_4 &\,=\, X_1+X_2+X_3 \, 
     \, ,&  \ell_5 &\,=\, X_1 + X_2+ Y_{23} \, ,& \ell_6 &\,=\, X_2 + X_3 + Y_{12} \, . \nonumber
\end{align} 
The set $\F = \{H_1,\dots,H_9\}$ contains nine elements that we list in \Cref{fig:adm}.
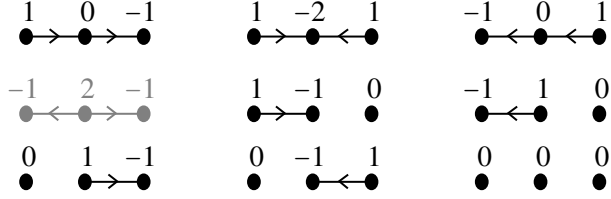
\begin{figure}[h]
  \begin{center}
\tikzset{every picture/.style={line width=0.75pt}} 

\begin{tikzpicture}[x=0.75pt,y=0.75pt,yscale=-1,xscale=.8,scale=0.92]

\draw    (40,40) -- (120,40) ;
\draw  [fill={rgb, 255:red, 0; green, 0; blue, 0 }  ,fill opacity=1 ] (36,40) .. controls (36,37.79) and (37.79,36) .. (40,36) .. controls (42.21,36) and (44,37.79) .. (44,40) .. controls (44,42.21) and (42.21,44) .. (40,44) .. controls (37.79,44) and (36,42.21) .. (36,40) -- cycle ;
\draw  [fill={rgb, 255:red, 0; green, 0; blue, 0 }  ,fill opacity=1 ] (116,40) .. controls (116,37.79) and (117.79,36) .. (120,36) .. controls (122.21,36) and (124,37.79) .. (124,40) .. controls (124,42.21) and (122.21,44) .. (120,44) .. controls (117.79,44) and (116,42.21) .. (116,40) -- cycle ;
\draw  [fill={rgb, 255:red, 0; green, 0; blue, 0 }  ,fill opacity=1 ] (76,40) .. controls (76,37.79) and (77.79,36) .. (80,36) .. controls (82.21,36) and (84,37.79) .. (84,40) .. controls (84,42.21) and (82.21,44) .. (80,44) .. controls (77.79,44) and (76,42.21) .. (76,40) -- cycle ;
\draw    (195,40) -- (275,40) ;
\draw  [fill={rgb, 255:red, 0; green, 0; blue, 0 }  ,fill opacity=1 ] (191,40) .. controls (191,37.79) and (192.79,36) .. (195,36) .. controls (197.21,36) and (199,37.79) .. (199,40) .. controls (199,42.21) and (197.21,44) .. (195,44) .. controls (192.79,44) and (191,42.21) .. (191,40) -- cycle ;
\draw  [fill={rgb, 255:red, 0; green, 0; blue, 0 }  ,fill opacity=1 ] (271,40) .. controls (271,37.79) and (272.79,36) .. (275,36) .. controls (277.21,36) and (279,37.79) .. (279,40) .. controls (279,42.21) and (277.21,44) .. (275,44) .. controls (272.79,44) and (271,42.21) .. (271,40) -- cycle ;
\draw  [fill={rgb, 255:red, 0; green, 0; blue, 0 }  ,fill opacity=1 ] (231,40) .. controls (231,37.79) and (232.79,36) .. (235,36) .. controls (237.21,36) and (239,37.79) .. (239,40) .. controls (239,42.21) and (237.21,44) .. (235,44) .. controls (232.79,44) and (231,42.21) .. (231,40) -- cycle ;
\draw    (350,40) -- (430,40) ;
\draw  [fill={rgb, 255:red, 0; green, 0; blue, 0 }  ,fill opacity=1 ] (346,40) .. controls (346,37.79) and (347.79,36) .. (350,36) .. controls (352.21,36) and (354,37.79) .. (354,40) .. controls (354,42.21) and (352.21,44) .. (350,44) .. controls (347.79,44) and (346,42.21) .. (346,40) -- cycle ;
\draw  [fill={rgb, 255:red, 0; green, 0; blue, 0 }  ,fill opacity=1 ] (426,40) .. controls (426,37.79) and (427.79,36) .. (430,36) .. controls (432.21,36) and (434,37.79) .. (434,40) .. controls (434,42.21) and (432.21,44) .. (430,44) .. controls (427.79,44) and (426,42.21) .. (426,40) -- cycle ;
\draw  [fill={rgb, 255:red, 0; green, 0; blue, 0 }  ,fill opacity=1 ] (386,40) .. controls (386,37.79) and (387.79,36) .. (390,36) .. controls (392.21,36) and (394,37.79) .. (394,40) .. controls (394,42.21) and (392.21,44) .. (390,44) .. controls (387.79,44) and (386,42.21) .. (386,40) -- cycle ;
\draw [color={gray}  ,draw opacity=1 ]   (40,82) -- (120,82) ;
\draw  [color={gray}  ,draw opacity=1 ][fill={rgb, 255:red, 128; green, 128; blue, 128 }  ,fill opacity=1 ] (36,82) .. controls (36,79.79) and (37.79,78) .. (40,78) .. controls (42.21,78) and (44,79.79) .. (44,82) .. controls (44,84.21) and (42.21,86) .. (40,86) .. controls (37.79,86) and (36,84.21) .. (36,82) -- cycle ;
\draw  [color={rgb, 255:red, 128; green, 128; blue, 128 }  ,draw opacity=1 ][fill={rgb, 255:red, 128; green, 128; blue, 128 }  ,fill opacity=1 ] (116,82) .. controls (116,79.79) and (117.79,78) .. (120,78) .. controls (122.21,78) and (124,79.79) .. (124,82) .. controls (124,84.21) and (122.21,86) .. (120,86) .. controls (117.79,86) and (116,84.21) .. (116,82) -- cycle ;
\draw  [color={rgb, 255:red, 128; green, 128; blue, 128 }  ,draw opacity=1 ][fill={rgb, 255:red, 128; green, 128; blue, 128 }  ,fill opacity=1 ] (76,82) .. controls (76,79.79) and (77.79,78) .. (80,78) .. controls (82.21,78) and (84,79.79) .. (84,82) .. controls (84,84.21) and (82.21,86) .. (80,86) .. controls (77.79,86) and (76,84.21) .. (76,82) -- cycle ;
\draw    (195,82) -- (235,82) ;
\draw  [fill={rgb, 255:red, 0; green, 0; blue, 0 }  ,fill opacity=1 ] (191,82) .. controls (191,79.79) and (192.79,78) .. (195,78) .. controls (197.21,78) and (199,79.79) .. (199,82) .. controls (199,84.21) and (197.21,86) .. (195,86) .. controls (192.79,86) and (191,84.21) .. (191,82) -- cycle ;
\draw  [fill={rgb, 255:red, 0; green, 0; blue, 0 }  ,fill opacity=1 ] (271,82) .. controls (271,79.79) and (272.79,78) .. (275,78) .. controls (277.21,78) and (279,79.79) .. (279,82) .. controls (279,84.21) and (277.21,86) .. (275,86) .. controls (272.79,86) and (271,84.21) .. (271,82) -- cycle ;
\draw  [fill={rgb, 255:red, 0; green, 0; blue, 0 }  ,fill opacity=1 ] (231,82) .. controls (231,79.79) and (232.79,78) .. (235,78) .. controls (237.21,78) and (239,79.79) .. (239,82) .. controls (239,84.21) and (237.21,86) .. (235,86) .. controls (232.79,86) and (231,84.21) .. (231,82) -- cycle ;
\draw    (350,82) -- (390,82) ;
\draw  [fill={rgb, 255:red, 0; green, 0; blue, 0 }  ,fill opacity=1 ] (346,82) .. controls (346,79.79) and (347.79,78) .. (350,78) .. controls (352.21,78) and (354,79.79) .. (354,82) .. controls (354,84.21) and (352.21,86) .. (350,86) .. controls (347.79,86) and (346,84.21) .. (346,82) -- cycle ;
\draw  [fill={rgb, 255:red, 0; green, 0; blue, 0 }  ,fill opacity=1 ] (426,82) .. controls (426,79.79) and (427.79,78) .. (430,78) .. controls (432.21,78) and (434,79.79) .. (434,82) .. controls (434,84.21) and (432.21,86) .. (430,86) .. controls (427.79,86) and (426,84.21) .. (426,82) -- cycle ;
\draw  [fill={rgb, 255:red, 0; green, 0; blue, 0 }  ,fill opacity=1 ] (386,82) .. controls (386,79.79) and (387.79,78) .. (390,78) .. controls (392.21,78) and (394,79.79) .. (394,82) .. controls (394,84.21) and (392.21,86) .. (390,86) .. controls (387.79,86) and (386,84.21) .. (386,82) -- cycle ;
\draw    (80,119) -- (120,119) ;
\draw  [fill={rgb, 255:red, 0; green, 0; blue, 0 }  ,fill opacity=1 ] (36,119) .. controls (36,116.79) and (37.79,115) .. (40,115) .. controls (42.21,115) and (44,116.79) .. (44,119) .. controls (44,121.21) and (42.21,123) .. (40,123) .. controls (37.79,123) and (36,121.21) .. (36,119) -- cycle ;
\draw  [fill={rgb, 255:red, 0; green, 0; blue, 0 }  ,fill opacity=1 ] (116,119) .. controls (116,116.79) and (117.79,115) .. (120,115) .. controls (122.21,115) and (124,116.79) .. (124,119) .. controls (124,121.21) and (122.21,123) .. (120,123) .. controls (117.79,123) and (116,121.21) .. (116,119) -- cycle ;
\draw  [fill={rgb, 255:red, 0; green, 0; blue, 0 }  ,fill opacity=1 ] (76,119) .. controls (76,116.79) and (77.79,115) .. (80,115) .. controls (82.21,115) and (84,116.79) .. (84,119) .. controls (84,121.21) and (82.21,123) .. (80,123) .. controls (77.79,123) and (76,121.21) .. (76,119) -- cycle ;
\draw    (235,119) -- (275,119) ;
\draw  [fill={rgb, 255:red, 0; green, 0; blue, 0 }  ,fill opacity=1 ] (191,119) .. controls (191,116.79) and (192.79,115) .. (195,115) .. controls (197.21,115) and (199,116.79) .. (199,119) .. controls (199,121.21) and (197.21,123) .. (195,123) .. controls (192.79,123) and (191,121.21) .. (191,119) -- cycle ;
\draw  [fill={rgb, 255:red, 0; green, 0; blue, 0 }  ,fill opacity=1 ] (271,119) .. controls (271,116.79) and (272.79,115) .. (275,115) .. controls (277.21,115) and (279,116.79) .. (279,119) .. controls (279,121.21) and (277.21,123) .. (275,123) .. controls (272.79,123) and (271,121.21) .. (271,119) -- cycle ;
\draw  [fill={rgb, 255:red, 0; green, 0; blue, 0 }  ,fill opacity=1 ] (231,119) .. controls (231,116.79) and (232.79,115) .. (235,115) .. controls (237.21,115) and (239,116.79) .. (239,119) .. controls (239,121.21) and (237.21,123) .. (235,123) .. controls (232.79,123) and (231,121.21) .. (231,119) -- cycle ;
\draw  [fill={rgb, 255:red, 0; green, 0; blue, 0 }  ,fill opacity=1 ] (346,119) .. controls (346,116.79) and (347.79,115) .. (350,115) .. controls (352.21,115) and (354,116.79) .. (354,119) .. controls (354,121.21) and (352.21,123) .. (350,123) .. controls (347.79,123) and (346,121.21) .. (346,119) -- cycle ;
\draw  [fill={rgb, 255:red, 0; green, 0; blue, 0 }  ,fill opacity=1 ] (426,119) .. controls (426,116.79) and (427.79,115) .. (430,115) .. controls (432.21,115) and (434,116.79) .. (434,119) .. controls (434,121.21) and (432.21,123) .. (430,123) .. controls (427.79,123) and (426,121.21) .. (426,119) -- cycle ;
\draw  [fill={rgb, 255:red, 0; green, 0; blue, 0 }  ,fill opacity=1 ] (386,119) .. controls (386,116.79) and (387.79,115) .. (390,115) .. controls (392.21,115) and (394,116.79) .. (394,119) .. controls (394,121.21) and (392.21,123) .. (390,123) .. controls (387.79,123) and (386,121.21) .. (386,119) -- cycle ;
\draw   (96,115) -- (102,118.75) -- (96,122.5) ;
\draw   (56,36) -- (62,39.75) -- (56,43.5) ;
\draw   (95,36) -- (101,39.75) -- (95,43.5) ;
\draw   (210,78) -- (216,81.75) -- (210,85.5) ;
\draw  [color={rgb, 255:red, 128; green, 128; blue, 128 }  ,draw opacity=1 ] (95,78) -- (101,81.75) -- (95,85.5) ;
\draw   (210,36) -- (216,39.75) -- (210,43.5) ;
\draw   (258.91,122.57) -- (253,118.68) -- (259.09,115.07) ;
\draw   (373.91,85.57) -- (368,81.68) -- (374.09,78.07) ;
\draw  [color={rgb, 255:red, 128; green, 128; blue, 128 }  ,draw opacity=1 ] (61.91,85.57) -- (56,81.68) -- (62.09,78.07) ;
\draw   (413.91,43.57) -- (408,39.68) -- (414.09,36.07) ;
\draw   (373.91,43.57) -- (368,39.68) -- (374.09,36.07) ;
\draw   (258.91,43.57) -- (253,39.68) -- (259.09,36.07) ;

\draw (35,19.4) node [anchor=north west][inner sep=0.75pt]    {$1$};
\draw (75,19.4) node [anchor=north west][inner sep=0.75pt]    {$0$};
\draw (105,19.4) node [anchor=north west][inner sep=0.75pt]    {$-1$};
\draw (190,19.4) node [anchor=north west][inner sep=0.75pt]    {$1$};
\draw (220,19.4) node [anchor=north west][inner sep=0.75pt]    {$-2$};
\draw (270,19.4) node [anchor=north west][inner sep=0.75pt]    {$1$};
\draw (335,19.4) node [anchor=north west][inner sep=0.75pt]    {$-1$};
\draw (385,19.4) node [anchor=north west][inner sep=0.75pt]    {$0$};
\draw (425,19.4) node [anchor=north west][inner sep=0.75pt]    {$1$};
\draw (25,61.4) node [anchor=north west][inner sep=0.75pt]  [color={gray}  ,opacity=1 ]  {$-1$};
\draw (75,61.4) node [anchor=north west][inner sep=0.75pt]  [color={gray}  ,opacity=1 ]  {$2$};
\draw (105,61.4) node [anchor=north west][inner sep=0.75pt]  [color={gray}  ,opacity=1 ]  {$-1$};
\draw (190,61.4) node [anchor=north west][inner sep=0.75pt]    {$1$};
\draw (220,61.4) node [anchor=north west][inner sep=0.75pt]    {$-1$};
\draw (270,61.4) node [anchor=north west][inner sep=0.75pt]    {$0$};
\draw (335,61.4) node [anchor=north west][inner sep=0.75pt]    {$-1$};
\draw (425,61.4) node [anchor=north west][inner sep=0.75pt]    {$0$};
\draw (385,61.4) node [anchor=north west][inner sep=0.75pt]    {$1$};
\draw (75,98.4) node [anchor=north west][inner sep=0.75pt]    {$1$};
\draw (35,98.4) node [anchor=north west][inner sep=0.75pt]    {$0$};
\draw (105,98.4) node [anchor=north west][inner sep=0.75pt]    {$-1$};
\draw (270,98.4) node [anchor=north west][inner sep=0.75pt]    {$1$};
\draw (190,98.4) node [anchor=north west][inner sep=0.75pt]    {$0$};
\draw (220,98.4) node [anchor=north west][inner sep=0.75pt]    {$-1$};
\draw (345,98.4) node [anchor=north west][inner sep=0.75pt]    {$0$};
\draw (385,98.4) node [anchor=north west][inner sep=0.75pt]    {$0$};
\draw (425,98.4) node [anchor=north west][inner sep=0.75pt]    {$0$};

\end{tikzpicture}
\caption{All oriented spanning trees for the $3$-site chain. The numbers represent the degree of each vertex. The one in gray is the only one which is not totally time-ordered in the sense of \Cref{def:tto}.
}
\label{fig:adm}
\end{center}
\end{figure}

\noindent We now construct the rational function in~\eqref{eq:rational_function} for the three oriented graphs in the left-most column of \Cref{fig:adm}. Let us denote them, from top to bottom, by $H_1,H_2,H_3$. Note that, as $H_3$ is disconnected, we denote $H_3^1,H_3^2$ its connected components. We have
\begin{align*}
\left(H_1\right)^+ \,=\, 
\biggl\{\treeConefirst \,,\,\treeCone\biggr\} \, ,
\ \left(H_2\right)^+ \,=\, \biggl\{ \treeCtwofirst\, ,\, \treeCtwosecond \, , \, \treeCtwothird\biggr\} \, ,
\ 
\left(H_3^2\right)^+ \,=\, \biggl\{\treeCthree\biggr\} \, , 
\end{align*}
and $(H_3^1)^+ = \varnothing$, which, using the definition in \eqref{eq:rational_function}, give
\[R_{H_1}=\,\frac{1}{\ell_4}\cdot\frac{1}{\ell_1\ell_5} \, , \quad
R_{H_2}=\,\frac{1}{\ell_4}\left(\frac{1}{\ell_2\ell_5}+\frac{1}{\ell_2\ell_6}\right) \, ,\quad
R_{H_3^1}\cdot R_{H_3^2} =\, \frac{1}{\ell_1} \left(\frac{1}{\ell_2}\cdot\frac{1}{\ell_6}\right) \, . \]
Note that the two summands in $R_{H_2}$ are given by the two possible choices of admissible sets given by first and third graph or first and third graph in~$(H_2)^+$.
The flat space wavefunction of the $3$-site chain is given by 
\begin{align*}
\psi^{\text{flat}}_3 \,=\, \frac{4Y_{12}Y_{23}(\ell_5+\ell_6)}{\ell_1 \ell_2\ell_3\ell_4\ell_5\ell_6} \ \in \,  \CC( X_1,X_2,X_3,Y_{12},Y_{23})  \, , 
\end{align*}
with $\ell_1,\ldots,\ell_6$ as in~\eqref{eq:ell_i_3site}, arising from~\Cref{eq:linear_forms_subgraphs}. The term $\ell_5+\ell_6$ in the numerator can be obtained via the recursion formula introduced in~\cite{cosmowave}; see also the appendix of \cite{FevolaMatsubara} for an explanation of this method. 

One can check that $\psi^{\text{flat}}_3$ decomposes as
\begin{align}\label{eq:PF_3site}\begin{split}
    \psi^{\text{flat}}_3 \ =\ \frac{1}{\ell_1\ell_4\ell_5} \,+\,\frac{1}{\ell_1\ell_3\ell_4} \,+\,\frac{1}{\ell_3\ell_4\ell_6} \,+\,
    \left(\frac{1}{\ell_2\ell_4\ell_5} \,+\,\frac{1}{\ell_2\ell_4\ell_6}\right) \, \\
    -\,
    \frac{1}{\ell_1\ell_3\ell_5} \,-\,\frac{1}{\ell_2\ell_3\ell_5} 
    \,-\,\frac{1}{\ell_1\ell_2\ell_6} \,-\,\frac{1}{\ell_1\ell_3\ell_6} \,+\,\frac{1}{\ell_1\ell_2\ell_3} 
    \, .
\end{split}\end{align}
In physics, this sum of ten terms naturally arises when computing the so-called ``bulk integrals'' constructed from a set of Feynman rules. 
For our current discussion, it suffices to know that, according to these rules, every vertex is associated to a time integral and physically represents an interaction between particles.
Every edge is represented by a ``bulk-to-bulk propagator'': a Green's function for the wave equation with cosmological boundary conditions. 
The propagator tells us how to order the two time integrals corresponding to the two vertices connected to the edge. Let $\nu_1$ and $\nu_2$ be two vertices. There are three options appearing as a sum in the bulk-to-bulk propagator: the vertex $\nu_1$
happens before $\nu_2$ in time, $\nu_2$ before $\nu_1$, or the two are unordered, i.e., disconnected. The integrand consists of a product of $n-1$ propagators which, after expanding, results in $3^{n-1}$ different terms. All of these different time-orderings are in one-to-one correspondence with the 
diagrams from the combinatorial approach allowing for a physical interpretation. An edge corresponds to a particle travelling between two vertices with the orientation of the edge describing the direction in time. In other words, all the oriented spanning graphs are different time-orderings in a physical process. 
\end{example}

\begin{example}[One-loop bubble]
In this example, we work out the partial fraction decomposition of the wavefunction in flat space for the genus-$1$ graph with two vertices, called ``one-loop bubble" in the physics literature. For an illustration of this graph, see the Introduction of~\cite{cosmowave}. The hyperplanes appearing in $\psi^{\text{flat}}_{1\text{-loop}}$ can be derived from the connected subgraphs by using \eqref{eq:linear_forms_subgraphs}. They are
\[
    \ell_1\,=\, X_1+Y+Y'\!,\, \ell_2 \,=\, X_2 + Y + Y'\!,\,   \ell_3 \,=\, X_1 + X_2 + 2Y' \!,\,
    \ell_4 \,=\, X_1+X_2+2Y  
     , \, \ell_5 \,=\, X_1 + X_2 .
\]
The set $\mathcal{F}$ contains nine elements, two of which are the leftmost graphs in \Cref{fig:loops}, which will not contribute in the decomposition, while the other ones are displayed in \Cref{fig:1loop}. 
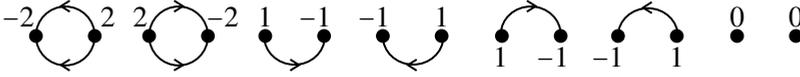
\begin{figure}
  \begin{center}
\tikzset{every picture/.style={line width=0.75pt}} 

\begin{tikzpicture}[x=0.75pt,y=0.75pt,yscale=-0.8,xscale=0.8,scale=0.92]

\draw  [fill={rgb, 255:red, 0; green, 0; blue, 0 }  ,fill opacity=1 ] (102,45) .. controls (102,42.79) and (103.79,41) .. (106,41) .. controls (108.21,41) and (110,42.79) .. (110,45) .. controls (110,47.21) and (108.21,49) .. (106,49) .. controls (103.79,49) and (102,47.21) .. (102,45) -- cycle ;
\draw  [fill={rgb, 255:red, 0; green, 0; blue, 0 }  ,fill opacity=1 ] (142,45) .. controls (142,42.79) and (143.79,41) .. (146,41) .. controls (148.21,41) and (150,42.79) .. (150,45) .. controls (150,47.21) and (148.21,49) .. (146,49) .. controls (143.79,49) and (142,47.21) .. (142,45) -- cycle ;
\draw   (123,21) -- (129,24.75) -- (123,28.5) ;
\draw   (123.1,60.93) -- (129,64.83) -- (122.91,68.42) ;
\draw   (106,45) .. controls (106,33.95) and (114.95,25) .. (126,25) .. controls (137.05,25) and (146,33.95) .. (146,45) .. controls (146,56.05) and (137.05,65) .. (126,65) .. controls (114.95,65) and (106,56.05) .. (106,45) -- cycle ;
\draw  [fill={rgb, 255:red, 0; green, 0; blue, 0 }  ,fill opacity=1 ] (25,45) .. controls (25,42.79) and (26.79,41) .. (29,41) .. controls (31.21,41) and (33,42.79) .. (33,45) .. controls (33,47.21) and (31.21,49) .. (29,49) .. controls (26.79,49) and (25,47.21) .. (25,45) -- cycle ;
\draw  [fill={rgb, 255:red, 0; green, 0; blue, 0 }  ,fill opacity=1 ] (65,45) .. controls (65,42.79) and (66.79,41) .. (69,41) .. controls (71.21,41) and (73,42.79) .. (73,45) .. controls (73,47.21) and (71.21,49) .. (69,49) .. controls (66.79,49) and (65,47.21) .. (65,45) -- cycle ;
\draw   (52,28.5) -- (46,24.75) -- (52,21) ;
\draw   (51.97,68.52) -- (46,64.73) -- (52.03,61.02) ;
\draw   (29,45) .. controls (29,33.95) and (37.95,25) .. (49,25) .. controls (60.05,25) and (69,33.95) .. (69,45) .. controls (69,56.05) and (60.05,65) .. (49,65) .. controls (37.95,65) and (29,56.05) .. (29,45) -- cycle ;
\draw  [fill={rgb, 255:red, 0; green, 0; blue, 0 }  ,fill opacity=1 ] (181,45) .. controls (181,42.79) and (182.79,41) .. (185,41) .. controls (187.21,41) and (189,42.79) .. (189,45) .. controls (189,47.21) and (187.21,49) .. (185,49) .. controls (182.79,49) and (181,47.21) .. (181,45) -- cycle ;
\draw  [fill={rgb, 255:red, 0; green, 0; blue, 0 }  ,fill opacity=1 ] (221,45) .. controls (221,42.79) and (222.79,41) .. (225,41) .. controls (227.21,41) and (229,42.79) .. (229,45) .. controls (229,47.21) and (227.21,49) .. (225,49) .. controls (222.79,49) and (221,47.21) .. (221,45) -- cycle ;
\draw   (202.02,60.99) -- (208,64.76) -- (201.98,68.49) ;
\draw  [fill={rgb, 255:red, 0; green, 0; blue, 0 }  ,fill opacity=1 ] (261,45) .. controls (261,42.79) and (262.79,41) .. (265,41) .. controls (267.21,41) and (269,42.79) .. (269,45) .. controls (269,47.21) and (267.21,49) .. (265,49) .. controls (262.79,49) and (261,47.21) .. (261,45) -- cycle ;
\draw  [fill={rgb, 255:red, 0; green, 0; blue, 0 }  ,fill opacity=1 ] (301,45) .. controls (301,42.79) and (302.79,41) .. (305,41) .. controls (307.21,41) and (309,42.79) .. (309,45) .. controls (309,47.21) and (307.21,49) .. (305,49) .. controls (302.79,49) and (301,47.21) .. (301,45) -- cycle ;
\draw   (287.98,68.51) -- (282,64.74) -- (288.02,61.01) ;
\draw  [fill={rgb, 255:red, 0; green, 0; blue, 0 }  ,fill opacity=1 ] (342,45) .. controls (342,42.79) and (343.79,41) .. (346,41) .. controls (348.21,41) and (350,42.79) .. (350,45) .. controls (350,47.21) and (348.21,49) .. (346,49) .. controls (343.79,49) and (342,47.21) .. (342,45) -- cycle ;
\draw  [fill={rgb, 255:red, 0; green, 0; blue, 0 }  ,fill opacity=1 ] (382,45) .. controls (382,42.79) and (383.79,41) .. (386,41) .. controls (388.21,41) and (390,42.79) .. (390,45) .. controls (390,47.21) and (388.21,49) .. (386,49) .. controls (383.79,49) and (382,47.21) .. (382,45) -- cycle ;
\draw   (363,21) -- (369,24.75) -- (363,28.5) ;
\draw  [fill={rgb, 255:red, 0; green, 0; blue, 0 }  ,fill opacity=1 ] (421,45) .. controls (421,42.79) and (422.79,41) .. (425,41) .. controls (427.21,41) and (429,42.79) .. (429,45) .. controls (429,47.21) and (427.21,49) .. (425,49) .. controls (422.79,49) and (421,47.21) .. (421,45) -- cycle ;
\draw  [fill={rgb, 255:red, 0; green, 0; blue, 0 }  ,fill opacity=1 ] (461,45) .. controls (461,42.79) and (462.79,41) .. (465,41) .. controls (467.21,41) and (469,42.79) .. (469,45) .. controls (469,47.21) and (467.21,49) .. (465,49) .. controls (462.79,49) and (461,47.21) .. (461,45) -- cycle ;
\draw   (448.03,29.48) -- (442,25.77) -- (447.97,21.98) ;
\draw  [fill={rgb, 255:red, 0; green, 0; blue, 0 }  ,fill opacity=1 ] (502,45) .. controls (502,42.79) and (503.79,41) .. (506,41) .. controls (508.21,41) and (510,42.79) .. (510,45) .. controls (510,47.21) and (508.21,49) .. (506,49) .. controls (503.79,49) and (502,47.21) .. (502,45) -- cycle ;
\draw  [fill={rgb, 255:red, 0; green, 0; blue, 0 }  ,fill opacity=1 ] (542,45) .. controls (542,42.79) and (543.79,41) .. (546,41) .. controls (548.21,41) and (550,42.79) .. (550,45) .. controls (550,47.21) and (548.21,49) .. (546,49) .. controls (543.79,49) and (542,47.21) .. (542,45) -- cycle ;
\draw  [draw opacity=0] (225.46,45.55) .. controls (225.46,45.6) and (225.46,45.65) .. (225.46,45.69) .. controls (225.3,56.25) and (215.95,64.67) .. (204.58,64.49) .. controls (193.21,64.32) and (184.12,55.62) .. (184.28,45.06) -- (204.87,45.38) -- cycle ; \draw   (225.46,45.55) .. controls (225.46,45.6) and (225.46,45.65) .. (225.46,45.69) .. controls (225.3,56.25) and (215.95,64.67) .. (204.58,64.49) .. controls (193.21,64.32) and (184.12,55.62) .. (184.28,45.06) ;  
\draw  [draw opacity=0] (306.18,45.49) .. controls (306.18,45.54) and (306.18,45.58) .. (306.18,45.63) .. controls (306.02,56.19) and (296.67,64.61) .. (285.3,64.43) .. controls (273.92,64.26) and (264.84,55.56) .. (265,45) -- (285.59,45.32) -- cycle ; \draw   (306.18,45.49) .. controls (306.18,45.54) and (306.18,45.58) .. (306.18,45.63) .. controls (306.02,56.19) and (296.67,64.61) .. (285.3,64.43) .. controls (273.92,64.26) and (264.84,55.56) .. (265,45) ;  
\draw  [draw opacity=0] (345.27,43.83) .. controls (345.27,43.78) and (345.27,43.73) .. (345.27,43.68) .. controls (345.43,33.19) and (354.67,24.82) .. (365.92,24.98) .. controls (377.17,25.15) and (386.17,33.79) .. (386.01,44.29) -- (365.64,43.99) -- cycle ; \draw   (345.27,43.83) .. controls (345.27,43.78) and (345.27,43.73) .. (345.27,43.68) .. controls (345.43,33.19) and (354.67,24.82) .. (365.92,24.98) .. controls (377.17,25.15) and (386.17,33.79) .. (386.01,44.29) ;  
\draw  [draw opacity=0] (425,45) .. controls (425,44.95) and (425,44.91) .. (425,44.86) .. controls (425.16,34.36) and (434.4,25.99) .. (445.65,26.16) .. controls (456.9,26.33) and (465.9,34.97) .. (465.74,45.46) -- (445.37,45.16) -- cycle ; \draw   (425,45) .. controls (425,44.95) and (425,44.91) .. (425,44.86) .. controls (425.16,34.36) and (434.4,25.99) .. (445.65,26.16) .. controls (456.9,26.33) and (465.9,34.97) .. (465.74,45.46) ;  

\draw (143,23.4) node [anchor=north west][inner sep=0.75pt]    {$-2$};
\draw (93,23.4) node [anchor=north west][inner sep=0.75pt]    {$2$};
\draw (4,23.4) node [anchor=north west][inner sep=0.75pt]    {$-2$};
\draw (71,23.4) node [anchor=north west][inner sep=0.75pt]    {$2$};
\draw (178,23.4) node [anchor=north west][inner sep=0.75pt]    {$1$};
\draw (206,23.4) node [anchor=north west][inner sep=0.75pt]    {$-1$};
\draw (246,23.4) node [anchor=north west][inner sep=0.75pt]    {$-1$};
\draw (298,23.4) node [anchor=north west][inner sep=0.75pt]    {$1$};
\draw (338,50.4) node [anchor=north west][inner sep=0.75pt]    {$1$};
\draw (367.64,50.4) node [anchor=north west][inner sep=0.75pt]    {$-1$};
\draw (405,50.4) node [anchor=north west][inner sep=0.75pt]    {$-1$};
\draw (458,50.4) node [anchor=north west][inner sep=0.75pt]    {$1$};
\draw (499,23.4) node [anchor=north west][inner sep=0.75pt]    {$0$};
\draw (539,23.4) node [anchor=north west][inner sep=0.75pt]    {$0$};

\end{tikzpicture}
\caption{The oriented spanning subgraphs for the one-loop bubble graph.}
\label{fig:1loop}
\end{center}
\end{figure}

\newpage 
\noindent Hence, the partial fraction decomposition is 
\begin{align*}
\psi^{\text{flat}}_{1\text{-loop}}\,=\, \frac{1}{\ell_2\ell_5}\,+\,\frac{1}{\ell_1\ell_5}\,-\,\frac{1}{\ell_1\ell_4}\,-\,\frac{1}{\ell_2\ell_4}\,-\,\frac{1}{\ell_1\ell_3}\,-\,\frac{1}{\ell_2\ell_3}\,+\,\frac{1}{\ell_1\ell_2} \, ,
\end{align*}
where the order of the terms matches the one of the spanning subgraphs in \Cref{fig:1loop}. The diagrams with at least one edge being deleted can be regarded as oriented spanning trees for the two-site chain after we redefine the variables associated to vertices. In physics, this recurrence is related to the so-called ``tree theorem''~\cite{CosmoTreeThm}. It is best illustrated by the example where we delete the upper edge from the bubble: 
\begin{equation*}
R \, \biggl(
\begin{tikzpicture}[scale=.419, transform shape,baseline=.1ex]
    \node[label=west:\huge{$X_1$},circle, draw, fill=black, scale=0.5] (X1) at (0,0.25) {X1};
    \node[label= east:\huge{$X_2$},circle, draw, fill=black, scale=0.5] (X2) at (2,0.25) {X2};
    \draw[thick] (0,0.25) arc (180:270:1);
    \draw[thick] (2,0.25) arc (180:90:-1);
    \draw[thick,->] (1.11,-.75)--(1.12,-.75);
    \node[draw=none,fill=none] at  (1.1,-0.1) {\huge $Y'$};
    \scoped[on background layer]
    \draw[thick,gray,dashed] (0,0.25) arc (180:0:1);
    \node[draw=none,fill=none,gray] at  (1.1,0.75) {\huge $Y$};
\end{tikzpicture}
\biggr) \, =\ 
R \, \biggl(
\begin{tikzpicture}[scale=.419, transform shape,baseline=.1ex]
    \node[label=north :\huge{$X_1+Y \ \  $},circle, draw, fill=black, scale=0.5] (X1) at (4,0) {X1}; 
\node[label=north:\huge{$\ \ X_2+Y$},circle, draw, fill=black, scale=0.5] (X2) at (6,0) {X2};
    \draw[thick,->] (4,0) -- (5.1,0);
    \draw[thick] (5,0) -- (6,0);
    \node[draw=none,fill=none] at (5,-0.7) {\huge$Y'$};
\end{tikzpicture}
\biggr)
\,=\,
\frac{1}{\ell_1 \ell_4} \, .
\end{equation*}
\end{example}

The primary motivation for applying partial fraction decomposition to the integrand is to facilitate and streamline the computation of the integral. The cosmological integral in \eqref{eq:cosmoint_explicit} would turn into the linear combination
\begin{equation}\label{eq:PF_integral}
    \sum_{H_i \,\in\, \mathcal{F}} (-1)^{n-1-|E(H_i)|} \ \cdot \! \int_{\RR^n_{>0}}  \prod_{j=1}^{s_i}  R_{H_i^j}(X+\alpha,Y)  \alpha^\eps \, \d \alpha \, .  
\end{equation}
The singularities of generalized Euler integrals with integrands involving only linear forms were studied in~\cite{FevolaMatsubara}. Specifically, the singularities characterize the locus where the Euler characteristic of the complement of the arrangement of hyperplanes in the integrand is smaller than the generic one.\footnote{This statement was communicated to us in conversations with S.-J.~Matsubara-Heo. It is getting addressed by him in ongoing work.}  
For real hyperplane arrangements, the signed Euler characteristic counts the number of bounded chambers.
Here, we focus on the complement of hyperplanes $\mathcal{H}^G_{(X,Y)}$ associated to cosmological integrals. In the case in which the coefficients of the hyperplanes vary in a subspace $Z$ of the coefficients space $\CC^A$, as in \Cref{sec:GKZ_restrictions}, the singular locus is described by the Euler discriminant~\cite{FMT24}. To give the definition, we first fix some notation. For each $z\in Z$, let 
\begin{align*}
\mathcal{H}^G(z)\, \coloneqq \, \left\{ \alpha\in(\CC^*)^n \,\, : \,\, L_i(z;\alpha) \neq 0, \  i=1,\dots,k \right\} \, ,
\end{align*}
and denote by $\chi_z$ the value of its signed Euler characteristic, $(-1)^n\cdot \chi(\mathcal{H}^G(z)).$

\pagebreak
\noindent The {\em Euler discriminant} is defined as the locus
\begin{align}\label{eq:Euler_discr}
\nabla_{\chi}(Z)\, = \, \{ z\in Z \,\, : \,\, \chi_z < \chi^*\}
\end{align}
for which the Euler characteristic is strictly smaller than its maximal value $\chi^* = \max_{z\in Z} \{\chi_z\}$. Note that $\nabla_{\chi}(Z)$ is a closed subvariety of $Z$, see \cite[Theorem~3.1]{FMT24}.

Let us represent each hyperplane $L_i$ as a scalar product between the vector $(\alpha,1)=(\alpha_1,\dots,\alpha_n,1)$ and the one describing the coefficients, denoted $T_i$. For instance, $L_1$ obtained from $\ell_1$ in~\eqref{eq:ell_i_3site} has the coefficient vector $T_1=(1,0,X_1+Y)$.

Consider the matrix $M'_G = [T_1 \,|\, \cdots\,|\, T_k]$ whose $i$-th column represents the coefficient vector of the hyperplane $L_i$. Let $M_G = [I_{n+1} \,|\, M'_G]$, where $I_{n+1}$ is the identity matrix of size~$n+1$. In~\cite{FevolaMatsubara}, it was proven that the Euler discriminant is a hypersurface and its defining polynomial $E_\chi(Z)$ can be expressed in terms of maximal minors of the matrix~$M_G$, namely
\begin{align*}
E_\chi(Z) \, = \, \prod_{\substack{I\subset [n+1+k],\,|I|=n+1\\ \det(M_G^I) \, \neq \, 0}} \det\big(M_G^I\big) \, ,
\end{align*}
where 
$M_G^I$ denotes the submatrix of $M_G$ whose columns are indexed by~$I$.

In~\cite{DEcosmological}, the singularities of the cosmological integrals of trees were read off from their differential equations in matrix form, and they can be described and characterized in terms of complete tubings of another type of decorated graphs. 
However, these singularities constitute a subset of the factors of the polynomial~$E_{\chi}(Z).$
In the physics literature, the linear factors occurring as entries of the $\operatorname{dlog}$'s in the connection matrix in $\eps$-factorized form as in~\cite[(3.30)]{DEcosmological} are called ``letters'' of the differential equation. In our example, all letters occur as factors of the singular locus. Moreover, all of these factors describe configurations of the kinematic variables for which the solution can develop a singularity that is expected from physical principles. These configurations are called \textit{physical singularities}, see e.g.~\cite{AnalyticWF}.

\begin{example}\label{ex:sing_3site}
    The singularities of the complements of hyperplanes $\mathcal{H}^G_{(X,Y)}$ associated to the $3$-site chain can be read from \cite[Example~5.4]{FMT24}. These in fact all correspond, up to a sign, to non-zero minors of the prescribed matrix $M_G$. However, the set of physical singularities are a subset of these, given by:
    \begin{align*}
X_1+&Y_{12},\,\,\,X_2+Y_{12}+Y_{23},\,\,\,X_3+Y_{23},\,\, \,X_1+X_2+X_3,\,\,\,X_1+X_2+Y_{23},\,\,\,X_2+X_3+Y_{12},\\
&\quad\quad\,\,\,\,\,X_1-Y_{12},\,\,\, X_2+Y_{12}-Y_{23},\,\,\, X_2-Y_{12}+Y_{23},\,\,\, X_2-Y_{12}-Y_{23},\\
&\qquad\qquad\qquad X_1+X_2-Y_{23},\,\,\, X_2+X_3-Y_{12},\,\,\,X_3-Y_{23}.
\end{align*}
We observed a connection between the terms in the partial fraction decomposition of the integral induced by~\eqref{eq:PF_3site} and this subset of singularities. 

In the remainder of this section, we will explain this relation in its full generality for arbitrary graphs.
\end{example}

\begin{definition}\label{def:tto}
    Given an element $H_i\in \mathcal{F}$, we say that $H_i$ is \textit{totally time-ordered} if the associated rational function $R_{H_i}$ is given by a single term.
\end{definition}
The motivation for the naming comes from cosmology: the totally time-ordered diagrams specify the time-ordering of all the vertices, and not just of the nearest neighbors. A counter-example would be the gray diagram in~\Cref{fig:adm}, for which the orientations are not sufficient to fix the ordering of the vertex $X_1$ with respect to~$X_3$. 
We observed for $n=2,3,4$ that all the physical singularities can be recovered by computing the singularities of the integrals in \eqref{eq:PF_integral} that correspond to totally time-ordered subgraphs. The example below provides a more explicit explanation of this observation. 

\begin{example}[\Cref{ex:sing_3site} continued]
    Let $n=3$. The partial fraction decomposition for the $3$-site chain is exhibited in \eqref{eq:PF_3site}. We can construct a matrix $[I_4\,\,| \,\, T_i\, T_j\, T_k]$ of size $4\times 7$ for each fraction of the form $1/(L_iL_jL_k)$ induced by the terms in \eqref{eq:PF_3site}. One can check that the physical singularities, displayed in \Cref{ex:sing_3site}, coincide with the minors of the matrices arising from the integrals associated to totally time-ordered graphs. 
\end{example}

\section{ Conclusion}\label{sec:outlook}
Cosmological integrals have interesting mathematical properties, and, despite having been studied in various guises over the past twenty years, have only recently been looked at systematically, in the physics literature. It is even earlier days in their mathematical exploration. In this article, we took some first steps in uncovering their properties, using some algebraic tools. 

Focusing on the integrals constructed in~\cite{DEcosmological}, we began to shed light on differential and difference equations satisfied by cosmological correlation functions, utilizing the Weyl algebra and shift algebras. For the two-site chain, we elaborated how the differential equations from~\cite{DEcosmological} can be obtained in terms of a restricted GKZ system. It deserves further study whether this connection holds true in greater generality, and it would be interesting to see which results from the theory of GKZ systems can be imported to handle cosmological integrals.

We constructed a multivariate version of partial fraction decomposition for the flat space wavefunction. This decomposition is described via spanning subgraphs of the modeling graph, which we conjecture to hold true for arbitrary graphs. This beautiful connection to subgraphs has an interesting physical interpretation in terms of spacetime locality of interactions between particles. We are likely just scratching the surface of an interesting connection between the wavefunction and decorated graphs---a theme that has appeared in various forms in the past few years. We also hope to be able to shed light on the observed ``dimension-drop'' in the matrix differential system of cosmological integrals. The dimension of the full space of master integrals is that of the top twisted cohomology space of the underlying very affine variety. In~\cite{DEcosmological}, it was observed that certain subspaces of master integrals can be chosen to obtain a closed subsystem of physically relevant ones; for the three-site graph, e.g., $16$ among originally~$25$. One should explore whether our method for partial fractioning of the flat space wavefunction can contribute to explaining this phenomenon.

There are several immediate directions for further exploration and development of our findings. It is clear that cosmological integrals still have many mathematical features yet to be fully understood, and the necessary tools to understand them will be very useful to physicists. Conversely, the examples from physics might suggest new mathematical structures, which are interesting and deserve of attention on their own right. 

\bigskip

\noindent {\bf Acknowledgements.} We thank Frédéric Chyzak for kindly providing an implementation of the algebraic Mellin transform in Maple, 
Daniel L.\ Bath, María Cruz Fernández-Fernández, Thomas Grimm, Arno Hoefnagels, Viktor Levandovskyy, Saiei-Jaeyeong Matsubara-Heo, and Uli Walther for insightful discussions, 
and Michael Borinsky, Anaëlle Pfister, and Jaroslav Trnka for their helpful comments to improve our manuscript.

\medskip
\noindent {\bf Funding statement.}
CF and ALS thank Scuola Normale Superiore for the generous support during a research visit. CF has received funding from the
European Union’s Horizon 2020 research and innovation programme under the Marie Sk\l odowska-Curie grant agreement No 101034255. 
The research of GLP and TW is supported by the ERC (NOTIMEFORCOSMO, 101126304), and that of ALS is funded by UNIVERSE+ (ERC, UNIVERSEPLUS, 101118787). Both projects are funded by the European Union  Views and opinions expressed are, however, those of the author(s) only and do not necessarily reflect those of the European Union or the European Research Council Executive Agency. Neither the European Union nor the granting authority can be held responsible for them. 
GLP and TW are moreover supported by Scuola Normale and by INFN (IS GSS-Pi). GLP is further supported by a Rita-Levi Montalcini Fellowship from the Italian Ministry of Universities and Research (MUR).

\begin{small}

\end{small}

\end{document}